\documentclass[12t,a4paper]{amsart}
\usepackage{mathtools,amssymb,amsmath,amsopn,amsthm,graphicx,MnSymbol,centernot,stmaryrd,array}
\usepackage{tikz-cd}
\usepackage{enumitem}
\usepackage{amsaddr}
\usepackage{etoolbox}
\usepackage[all,curve,frame]{xy}
\patchcmd{\subsection}{-.5em}{.5em}{}{}
\usetikzlibrary{matrix}
\makeatletter
\newcommand*{\doublerightarrow}[2]{\mathrel{
  \settowidth{\@tempdima}{$\scriptstyle#1$}
  \settowidth{\@tempdimb}{$\scriptstyle#2$}
  \ifdim\@tempdimb>\@tempdima \@tempdima=\@tempdimb\fi
  \mathop{\vcenter{
    \offinterlineskip\ialign{\hbox to\dimexpr\@tempdima+1em{##}\cr
    \rightarrowfill\cr\noalign{\kern.5ex}
    \rightarrowfill\cr}}}\limits^{\!#1}_{\!#2}}}
\newcommand*{\triplerightarrow}[1]{\mathrel{
  \settowidth{\@tempdima}{$\scriptstyle#1$}
  \mathop{\vcenter{
    \offinterlineskip\ialign{\hbox to\dimexpr\@tempdima+1em{##}\cr
    \rightarrowfill\cr\noalign{\kern.5ex}
    \rightarrowfill\cr\noalign{\kern.5ex}
    \rightarrowfill\cr}}}\limits^{\!#1}}}
\newcommand{\nd}[1]{\begin{smallmatrix}#1\end{smallmatrix}}
\makeatother
\newcommand{\uwithtext}[1]{\DOTSB\mathbin{\text{\tikz{%
    \node (a) {#1};
    \draw[black, rounded corners=1.2ex] 
        ([yshift=-2pt]a.north west) -- (a.south west) -- (a.south east) -- ([yshift=-2pt]a.north east);}}}}

\begin{document}

\newtheorem{definition}{Definition}[section]
\newtheorem{definitions}[definition]{Definitions}
\newtheorem{deflem}[definition]{Definition and Lemma}
\newtheorem{lemma}[definition]{Lemma}
\newtheorem{proposition}[definition]{Proposition}
\newtheorem{theorem}[definition]{Theorem}
\newtheorem{corollary}[definition]{Corollary}
\newtheorem{algo}[definition]{Algorithm}
\theoremstyle{remark}
\newtheorem{rmk}[definition]{Remark}
\theoremstyle{remark}
\newtheorem{remarks}[definition]{Remarks}
\theoremstyle{remark}
\newtheorem{notation}[definition]{Notation}
\newtheorem{assumption}[definition]{Assumption}
\theoremstyle{remark}
\newtheorem{example}[definition]{Example}
\theoremstyle{remark}
\newtheorem{examples}[definition]{Examples}
\theoremstyle{remark}
\newtheorem{dgram}[definition]{Diagram}
\theoremstyle{remark}
\newtheorem{fact}[definition]{Fact}
\theoremstyle{remark}
\newtheorem{illust}[definition]{Illustration}
\theoremstyle{remark}
\newtheorem{que}[definition]{Question}
\theoremstyle{definition}
\newtheorem{conj}[definition]{Conjecture}
\newtheorem{scho}[definition]{Scholium}
\newtheorem{por}[definition]{Porism}
\DeclarePairedDelimiter\floor{\lfloor}{\rfloor}

\renewenvironment{proof}{\noindent {\bf{Proof.}}}{\hspace*{3mm}{$\Box$}{\vspace{9pt}}}
\author[Sardar]{Shantanu Sardar}
\address{Department of Mathematics \\Universidad Nacional de Mar del Plata, Mar del Plata\\ Buenos Aires, Argentina}
\email{shantanusardar17@gmail.com}
\title{{Krull-Gabriel dimension of Skew group algebras}}
\subjclass[2020]{}

\begin{abstract}
For an algebraically closed field $K$, let $G$ be a finite abelian group of $K$-linear automorphisms of a finite-dimensional algebra $\Lambda$ and $\bar\Lambda(=\Lambda G)$ is the associated skew group algebra. The author with S. Trepode and A. G. Chaio introduced the notion of a Galois semi-covering functor to study the irreducible morphisms over skew group algebras. In this paper, we establish a Galois semi-covering functor between the morphism categories as well as the functor categories over the algebras $\Lambda$ and $\bar\Lambda$ and prove that their Krull-Gabriel dimension are equal. This computation confirms Prest's conjecture on the finiteness of Krull-Gabriel dimension and Schr\"oer's conjecture on its connection with the stable rank (the least stabilized radical power) over skew gentle algebras. Moreover, we determine all posible stable ranks for (skew) Brauer graph algebras.  
\end{abstract}

\maketitle
\section{Introduction}
The Krull–Gabriel dimension, $\mathrm{KG}(\Lambda)$, of a ring was introduced as a modification of the Gabriel dimension for the category of finitely presented modules by Geigle in his thesis \cite{Ge85}. This dimension also turns out \cite[Lemma~B.9]{Krau01} to be the $m$-dimension of the lattice of pp formulas for $\Lambda$, equivalently of the lattice of pointed finite-dimensional modules; that dimension is a variant of the elementary Krull dimension, introduced by Garavaglia \cite{Ga80}. The Krull-Gabriel dimension is defined for any locally bounded K-category. In fact, it is defined for any skeletally small abelian category.

\begin{notation}
Assume that $Q=(Q_0,Q_1,s,t)$ is a finite quiver, where $Q_0$ is the set of vertices, $Q_1$ is the set of arrows and $s,t: Q_1\to Q_0$ denote the source and target functions. Fix $\tilde I$ to be a set of representatives of $Q_0$ under the action of G. Set $G_{i_0}$ as the stabilizer of $i_0\in \tilde I$ in $G$ and $G_{i_0j_0}:=G_{i_0}\cap G_{j_0}$. The $G$-orbit of a path $p$ is denoted by $O_p$. We denote by $\mathrm{mod}\mbox{-} K$ the category of all finite-dimensional $K$-vector spaces and by $\mathrm{mod}\mbox{-}\Lambda$ the category of all finitely generated right $\Lambda$-modules. We write $M\inplus N$ for $M, N\in \mathrm{mod}\mbox{-}\Lambda$ if $M$ is a direct summand of $N$.  Write $\operatorname{rad}^\alpha_A$ as the transfinite powers of the radical $\operatorname{rad}_A$ of $\mathrm{mod}\mbox{-}\Lambda$. A subfunctor $G$ of a functor $F$ is denoted by $G\subseteq F$ and the corresponding quotient functor is denoted by $F/G$.
\end{notation}

We start with some well-known results about the KG dimension. M. Auslander proves in \cite[Corollary~3.14]{Au82} that $\mathrm{KG}(\Lambda)= 0$ if and only if the algebra $\Lambda$ is of finite representation type. H. Krause shows in \cite[11.4]{Kr98} that $\mathrm{KG}(\Lambda)\neq 1$ for any algebra $\Lambda$. W. Geigle proves in \cite[4.3]{Ge86} that if $\Lambda$ is a tame hereditary algebra, then $\mathrm{KG}(\Lambda)=2$. A.Skowronski shows in \cite[Theorem~1.2]{Sk16} that if $\Lambda$ is a cycle-finite algebra of domestic representation type, then $\mathrm{KG}(\Lambda)= 2$. M. Wenderlich proves in \cite{We96} that if $\Lambda$ is a strongly simply connected algebra, then $\Lambda$ is of domestic type if and only if $\mathrm{KG}(\Lambda)$ is finite; R. Laking, M. Prest and G. Puninski prove in \cite{LaPrPu18} a renowned result that string algebras of domestic representation type have finite KG dimension. H. Krause proves in \cite[Corollary~8.14]{Krau01} that if $\mathrm{KG}(\Lambda)= n \in \mathbf{N}$, then $\mathrm{rad}^{\omega(n+1)}_\Lambda=0$ i.e. their stable rank is at most $\omega(n+1)$ where $\omega$ is the first infinite cardinal. Here is a list of algebras where the width of the lattice of all pointed modules is infinite and as a result, their KG dimension is undefined. This list includes (strictly) wild algebras, nondomestic string algebras, tubular algebras, pg-critical algebras, strongly simply connected algebras of nonpolynomial growth and weighted surface algebras.

Classifying a finite-dimensional associative algebra as finite, tame (domestic and non-domestic), or wild is a challenging problem. The computation of some numerical invariants like KG dimension and (stable) rank could make this task easier as we see in the last paragraph. Here, we are concerned about the KG dimension of skew group algebras. These kind of algebras were first studied from the point of view of representation theory in \cite{ReRi85}. For a finite-dimensional algebra $\Lambda$ over a field $K$ and a finite group $G$ acting on $\Lambda$ by automorphisms, the skew group algebra $\bar\Lambda$ shares many representation-theoretic properties with $\Lambda$ often incarnated in properties of functors between $\mathrm{mod}\mbox{-}\Lambda$ and $\mathrm{mod}\mbox{-}\bar\Lambda$.  


\begin{definition}[Skew group algebras]
Let $G$ be a finite group acting on an algebra $\Lambda$ by automorphisms. The skew group algebra $\bar\Lambda$ is the algebra defined by:
\begin{itemize}
    \item its underlying vector space is $\Lambda \otimes_K KG$;
    \item multiplication is given by $(\lambda \otimes g)(\mu \otimes h) = \bar\Lambda(\mu) \otimes gh$ for $\lambda,\mu\in\Lambda$ and $g,h\in G$, extended by linearity and distributivity.
\end{itemize}
\end{definition}
Notice that the algebra $\bar\Lambda$ is not basic in general.

The author with S. Trepode and A. G. Chaio \cite{GoSaTr25} showed the existence of certain kind of functor known as a Galois semi-covering (Section \ref{SGCM}) between the module categories of $\Lambda$ and $\bar{\Lambda}$ to study the irreducible morphisms over skew group algebras. In this paper, we consider the morphism category $\mathrm{H}(\mathrm{mod}\mbox{-}\Lambda)$ where the objects are the homomorphisms over the module category and morphisms are given by a pair of homomorphisms satisfying a commutative diagram (Section \ref{MGC}). Applying a functor component-wise we extend the functor $F_\lambda: \mathrm{mod}\mbox{-}\Lambda\to \mathrm{mod}\mbox{-}\bar{\Lambda}$ to the functor $HF_\lambda: \mathrm{H}(\mathrm{mod}\mbox{-}\Lambda) \to \mathrm{H}(\mathrm{mod}\mbox{-}\bar{\Lambda})$. We demonstrate that this induced functor is indeed a Galois semi-covering functor. 

\begin{theorem}
Assume that a group $G$ acts on an algebra $\Lambda$ and $\bar\Lambda$ is the associated skew group algebra. Then for any $f, h\in \mathrm{H}(\mathrm{mod}\mbox{-}\Lambda)$, the functor $\mathrm{H}F_\lambda: \mathrm{H}(\mathrm{mod}\mbox{-}\Lambda) \to \mathrm{H}(\mathrm{mod}\mbox{-}\bar\Lambda)$ induces the following isomorphisms of vector spaces:
$$\mathrm{H}(\mathrm{mod}\mbox{-}\bar\Lambda)(\mathrm{H}F_\lambda f,\mathrm{H}F_\lambda h)\approx \begin{cases}\bigoplus_{g\in G} \mathrm{H}(\mathrm{mod}\mbox{-}\Lambda)({}^gf, h)&\mbox{ if } G_{f}\neq G;\\\bigoplus_{g\in G} \mathrm{H}(\mathrm{mod}\mbox{-}\Lambda)(f, {}^gh)&\mbox{ if }G_{h}\neq G;\\\mathrm{H}(\mathrm{mod}\mbox{-}\Lambda)^{|G|}(f, h)&\mbox{ if }G_{fh}= G.\end{cases}$$  
\end{theorem}

We also study the categories $\mathcal{F}(\Lambda)$ and $\mathcal{F}(\bar{\Lambda})$ of finitely presented functors over $\Lambda$ and $\bar{\Lambda}$ respectively and introduce an exact and faithful functor $\phi: \mathcal{F}(\Lambda) \to \mathcal{F}(\bar\Lambda)$ which is also found to be a Galois semi-covering functor. 
\begin{theorem}
Suppose a finite abelian group $G$ acts on a finite-dimensional algebra $\Lambda$ with $\bar\Lambda$ the associated skew group algebra. Then for any $T_1, T_2 \in \mathcal{F}(\Lambda)$, the functor $\phi: \mathcal{F}(\Lambda) \to \mathcal{F}(\bar\Lambda)$ induces the following isomorphisms of vector spaces:
$$\mathcal{F}(\bar\Lambda)(\phi(T_1), \phi(T_2))\approx \begin{cases}\bigoplus_{g\in G} \mathcal{F}(\Lambda)({}^gT_1, T_2)&\mbox{ if } G_{T_1}\neq G;\\\bigoplus_{g\in G} \mathcal{F}(\Lambda)(T_1, {}^gT_2)&\mbox{ if }G_{T_2}\neq G;\\\mathcal{F}(\Lambda)^{|G|}(T_1, T_2)&\mbox{ if }G_{T_1T_2}= G.\end{cases}$$ 
\end{theorem}

In \cite{GoSaTr25}, we established that the stable rank is preserved under skew group algebra construction. The main contribution in this paper is that skewness also preserves the KG dimension.

\begin{theorem}
Suppose $\Lambda$ is a finite-dimensional algebra, $G$ is a finite abelian group of $K$-linear automorphisms of $\Lambda$ and $\bar\Lambda$ is the associated skew group algebra. Then, the KG dimension $\mathrm{KG}(\bar\Lambda)$ of $\bar\Lambda$ equals the KG dimension $\mathrm{KG}(\Lambda)$ of $\Lambda$.
\end{theorem}

As an application, we determine the KG dimension of skew-gentle algebras and skew-Brauer algebras. Skew-gentle algebras were introduced in \cite{GdlP99} to study a class of tame matrix problems, and appears as a skew group algebra of gentle algebras for the group of order two when the characteristic of the field is different from two. We also verify the following well known conjectures regarding the relation between the KG dimension and the nilpotent index of radical for skew-gentle algebras.
\begin{corollary}\label{conjsp} The following is a list of conjectures:
\begin{enumerate}
\item[J. Schroer] \cite{Sc00} $\mathrm{KG}(\Lambda) = n \geq 2$ if and only if $\mathrm{rad}^{\omega(n-1)}_\Lambda=0$ and $\mathrm{rad}^{\omega n}_\Lambda=0$;

\item[M. Prest] \cite{Pr88} \cite{Pr09} A finite-dimensional algebra $\Lambda$ is of domestic representation type if and only if the Krull-Gabriel dimension $\mathrm{KG}(\Lambda)$ of $\Lambda$ is finite.
\end{enumerate}    
\end{corollary}
On the other hand, Brauer graph algebras, a special subclass of special biserial algebras that consists of symmetric algebras \cite[Section~2.3]{Sc18}, originate in the modular representation theory of finite groups in the form of Brauer tree algebras \cite{Ja69}. The interpretation of Brauer graphs as ribbon graphs \cite{MaSc14} links Brauer graph algebras with surface cluster theory and adds a geometric perspective to the representation theory of Brauer graph algebras. Skew Brauer graph algebras are generalization of Brauer graph algebras \cite[Section~3.1]{So24}. It is shown in \cite[Proposition~3.9]{So24} that for a skew Brauer graph algebra $\bar{\Lambda}$ one can construct a Brauer graph algebra $\Lambda$ such that $\bar{\Lambda}$ appears as a skew group algebra of $\Lambda$ with a certain action of the group of order two. As a consequence, their KG dimensions are equal. Moreover, we demonstrate the following result regarding their stable ranks.

\begin{theorem}\label{stablesba}
For a skew Brauer graph algebra $\bar{\Lambda}$ and its associated Brauer graph algebra $\Lambda$, we have $\omega \leq \mathrm{st}(\Lambda)=\mathrm{st}(\bar{\Lambda}) < \omega^2$.    
\end{theorem}

We arrange the paper as follows. We start with the basic information about skew group algebras in Section $(2)$. The main highlight of this section is the existence of the Galois semi-covering functor $F_\lambda: \mathrm{mod}\mbox{-}\Lambda \to  \mathrm{mod}\mbox{-}\bar\Lambda$, which is a pushdown of a functor $F:\Lambda \to \bar\Lambda$. As $\Lambda$ also appears as a skew group algebra of $\bar{\Lambda}$ under certain action of the dual group (Theorem \ref{inv}), we obtain another Galois semi-covering functor $F^\lambda: \mathrm{mod}\mbox{-}\bar\Lambda \to  \mathrm{mod}\mbox{-}\Lambda$. In Section $(3)$, we first show that $F_\lambda$ and $F^\lambda$ are adjoint to each other (Proposition \ref{gcad}). We also establish the existence of the Galois semi-covering functor $\mathrm{H}F_\lambda: \mathrm{H}(\mathrm{mod}\mbox{-}\Lambda) \to \mathrm{H}(\mathrm{mod}\mbox{-}\bar\Lambda)$ (resp. $\mathrm{S}F_\lambda: \mathrm{S}(\mathrm{mod}\mbox{-}\Lambda) \to \mathrm{S}(\mathrm{mod}\mbox{-}\bar\Lambda)$) between the morphism (resp. monomorphism) categories of an algebra and its skew group algebras (Theorem \ref{HGCM}) with an illustrative Example \ref{GCMOR}; where the functors $\mathrm{H}F_\lambda$ and $\mathrm{H}F^\lambda$ (resp. $\mathrm{S}F_\lambda$ and $\mathrm{S}F^\lambda$) are $G$-stable on objects (Proposition \ref{GSTAB}) and adjoint to each other (Lemma \ref{HADJ}). In Section $(4)$, we consider the category $\mathcal{F}(\Lambda)$ (resp. $\mathcal{F}(\bar{\Lambda})$) of finitely presented functors over the module category of $\Lambda$ (resp. $\bar{\Lambda}$) and present an exact functor $\phi: \mathcal{F}(\Lambda) \to \mathcal{G}(\bar\Lambda)$ (Propositions \ref{exctphi} and \ref{cokphi}) which is later shown as a Galois semi-covering faithful functor (Theorem \ref{GCF} and Corollary \ref{faithful}). In Section $(5)$, we prove that the KG dimension remain same under skew group algebra construction (Theorem \ref{eqkg}) and as an application, the KG dimension of a gentle (rep. Brauer graph) algebra and its skew gentle (Brauer graph) algebra are equal (Corollaries \ref{eqkgg} and \ref{eqkgb}). Corollary \ref{domfi} confirms the Prest's conjecture that states the domesticity of skew gentle algebras are characterized by the finiteness of their KG dimension and Corollary \ref{conjs} confirms the Schr\"oer's conjecture for skew gentle algebras. Moreover, we determine the all possible stable ranks for (skew) Brauer graph algebras (Theorem \ref{srbg}).  

\subsection*{Acknowledgements}
The author had a stay at the CEMIM-Universidad Nacional de Mar del Plata, supported by a postdoctoral scholarship from CONICET.

\begin{section}{Galois semi-covering Functor over the module category}\label{SGCM}
In the PhD thesis \cite{Jose83}, J. A. de la Pe\~na observed that the concept of skewness is identical with the Galois covering in the case of a free group action. This motivated us to study skew group algebra construction when the group action is not necessarily free and introduce the notion of the Galois semi-covering functor. Let us first recall the Galois semi-covering $F_\lambda: \mathrm{mod}\mbox{-}\Lambda \to  \mathrm{mod}\mbox{-}\bar\Lambda$ from \cite{GoSaTr25} to define the functor $\phi$ which plays the fundamental role in preserving the KG dimension.

\begin{definition}
Let $A,B$ be linear categories with $G$ a group acting on $A$. A functor $F: A\to B$ is called a $G$-semi covering if for any $X,Y\in Ob(A)$, the following hold:
$$B(FX,FY)\approx \begin{cases}\bigoplus_{g\in G} A(gX,Y)&\mbox{if } G_{X}\neq G;\\\bigoplus_{g\in G} A(X,gY)&\mbox{if }G_{Y}\neq G;\\A^{|G|}(X,Y)&\mbox{ if }G_{XY}= G.\end{cases}$$    
\end{definition}

\begin{assumption}
Let us fix a finite abelian group $G$ of order $n$ acting on $KQ$ such that $G_{i_0}\neq G$ implies $|O_{i_0}|=n$ for each $i_0\in \tilde I$. Recall that the set of all irreducible representations of $G$, denoted $\hat{G}$, forms a group w.r.t. tensor product of representations with $|\hat{G}|=n$.
\end{assumption}

The vertices of $Q_G$, the underlying quiver of the skew group algebra, are given by
$Q_{G_0}=\{(i_0, \rho)\mid i_0 \in \tilde I, \rho \in \hat{G}_{i_0}\}.$ The idempotent of $(KQ)G$ corresponding to the vertex $(i_0, \rho)$ is $e_{i_0\rho} = i_0\otimes e_\rho \mbox{, where } e_\rho = \frac{1}{|G_{i_0}|} \sum_{g\in G_{i_0}} \rho(g)g$ is an idempotent of $KG_{i_0}$. Consider the idempotent of $KQ_G$
$\bar e = \sum_{i_0\in \tilde I} \bar {e_{i_0}} \mbox{ where } \bar {e_{i_0}} = \sum_{\rho \in \hat{G}_{i_0}}e_{i_0\rho}.$

We recall the semi-covering functor \cite[Proposition~2.12]{GoSaTr25} induced by $F: KQ \to (KQ)G$ given by setting for each $i \in Q_0$, $F(i)=\bar {e_{i_0}} \mbox{ where, } i_0\in O_i\cap \tilde I$ and its push-down functor, $F_\lambda: \mathrm{mod}\mbox{-}\Lambda \to \mathrm{mod}\mbox{-}\bar\Lambda$. 

\begin{definition}\label{dgsc}
Let $F: \Lambda \to \bar\Lambda$ be the Galois semi-covering functor. We define the pushdown functor $F_\lambda: \mathrm{mod}\mbox{-}\Lambda \to \mathrm{mod}\mbox{-}\bar\Lambda$ as follows:

Suppose $M \in \mathrm{mod}\mbox{-}\Lambda$, then $F_\lambda M= \bigoplus_{i_0\in \tilde I} F_\lambda M(\bar {e_{i_0}})$ where, for each $\bar {e_{i_0}}\in \bar\Lambda$, we set
$$F_\lambda M(\bar {e_{i_0}}):= \bigoplus_{F(x)=\bar {e_{i_0}}}M(x).$$

Assume that $\bar{\alpha}\in \Lambda_1G(\bar {e_{i_0}},\bar {e_{j_0}})$. Consider the following cases:
\begin{enumerate}
\item If $G_{i_0}\neq G$ then there is an arrow $\alpha_h: h{i_0}\to {j_0}$ for some $h\in G$ such that $\bar{\alpha}=F(\alpha_h)$. Then the homomorphism $F_\lambda M(\bar{\alpha}): F_\lambda M(e_{i_0}) \to F_\lambda M(\bar {e_{j_0}})$ is defined by homomorphism:
$$(\mu_g) \mapsto (\sum_{g\in G} M(g\alpha_h)(\mu_g)).$$

\item If $G_{i_0}=G$ but $G_{j_0}\neq G$ then there is an arrow $\alpha_h: i_0\to hj_0$ for some $h\in G$ such that $\bar{\alpha}=F(\alpha_h)$. Then the homomorphism $F_\lambda M(\bar{\alpha}): F_\lambda M(\bar {e_{i_0}}) \to F_\lambda M(e_{j_0})$ is defined by:
$$\mu \mapsto (M(g_1\alpha_h)(\mu),\hdots, M(g_n\alpha_h)(\mu)).$$

\item If $G_{i_0j_0}= G$ then there is an arrow $\alpha: i_0\to j_0$ such that $\bar{\alpha}=F(\alpha)$. The homomorphism $F_\lambda M(\bar{\alpha}): F_\lambda M(\bar {e_{i_0}}) \to F_\lambda M(\bar {e_{j_0}})$ is defined by homomorphism:
$$\mu \mapsto M(\alpha)(\mu).$$
\end{enumerate}
\end{definition}


 

\vspace{0.1in}
\textbf{Morphism in $\mathrm{Mod}\mbox{-}\bar\Lambda$:} 
Assume that $f:=(f_{i_0})_{i_0\in Q_0}: M\to N \in \mathrm{mod}\mbox{-}\Lambda$ is a module homomorphism where, $f_{i_0}: M(i_0) \to N(i_0)$. Then $F_\lambda(f):= (\hat{f}_{\bar{e_{i_0}}}): F_\lambda(M) \to F_\lambda(N)$ where, $\hat{f}_{\bar{e_{i_0}}}: F_\lambda(M)(\bar{e_{i_0}}) \to F_\lambda(N)(\bar{e_{i_0}})$ is defined by homomorphisms $f_{i}: M(i) \to N(i)$, for all $i\in \mathcal{O}(i_0)$. If $M\in \mathrm{mod}\mbox{-}\Lambda$, then $F_\lambda(M)\in \mathrm{mod}\mbox{-}\bar\Lambda$. Hence, the functor $F_\lambda$ restricts to a functor $\mathrm{mod}\mbox{-}\Lambda \to \mathrm{mod}\mbox{-}\bar\Lambda$, also denoted by $F_\lambda$. 

\vspace{0.1in}
\textbf{Group action over $\mathrm{Mod}\mbox{-}\Lambda$:} Given a $M\in \mathrm{mod}\mbox{-}\Lambda$ we denote by ${}^gM$ the module $M\circ g^{-1}$ and a module homomorphism $f: M \to N$ we denote by ${}^gf$ the $\Lambda$-module homomorphism ${}^gM \to {}^gN$ such that ${}^g{f_x}:= f_{g^{-1}x}$, for any $x \in Q_0$. This defines an action of $G$ on $\mathrm{Mod}\mbox{-}\Lambda$. Moreover, the map $f\to {}^gf$ defines isomorphism of vector spaces $\Lambda(M,N) \approx \Lambda({}^gM, {}^gN)$. 

The following theorem ensures the existence of a Galois semi-covering from $\mathrm{mod}\mbox{-}\Lambda$ to $\mathrm{mod}\mbox{-}\bar\Lambda$.

\begin{theorem}\label{GCM}\cite[Theorem~3.11]{GoSaTr25}
Assume that a group $G$ acts on an algebra $\Lambda$ and $F: \Lambda \to \bar\Lambda$ is a Galois semi-covering. Then for any $M, N\in \mathrm{mod}\mbox{-}\Lambda$, the functor $F_\lambda: \mathrm{mod}\mbox{-}\Lambda \to \mathrm{mod}\mbox{-}\bar\Lambda$ induces the following isomorphisms of vector spaces:
$$\mathrm{Hom}_{\bar\Lambda}(F_\lambda M,F_\lambda N)\approx \begin{cases}\bigoplus_{g\in G} \mathrm{Hom}_\Lambda ({}^gM,N)&\mbox{ if } G_{M}\neq G;\\\bigoplus_{g\in G} \mathrm{Hom}_\Lambda(M,{}^gN)&\mbox{ if }G_{N}\neq G;\\\mathrm{Hom}_\Lambda^{|G|}(M,N)&\mbox{ if }G_{MN}= G.\end{cases}$$  
\end{theorem}

The following remark \cite{GoSaTr25} show that, unlike Galois covering, Galois semi-covering does not necessarily preserve indecomposable modules.
\begin{rmk}\label{indsk}
For all ${g\in G}$ there exist a $\bar M \in \mathrm{Ind}\mbox{-}\bar\Lambda$ such that,
$$F_\lambda (^gM)= \begin{cases}\bar M &\mbox{ if } G_{M}\neq G;\\\bigoplus_{\hat{g}\in \hat{G}} \hat{g} \bar M &\mbox{ if }G_{M}= G.\end{cases}$$ 
\end{rmk}

The following theorem \cite[Corollary~5.2]{ReRi85} shows the involutive act of skewness.
\begin{theorem}\label{inv}
For an abelian group $G$, the algebra $\Lambda$ is Morita equivalent with the skew group algebra of $(\bar\Lambda)\hat{G}$, where the action of $\hat{G}$ on $\bar\Lambda$ is defined by $\chi (\lambda \otimes g): = \chi(g) \lambda \otimes g$ for $\lambda \in \Lambda, g \in G$.  
\end{theorem}

Since the algebra $\Lambda$ appears as the skew algebra of $\bar{\Lambda}$ under the action of the dual group $\hat{G}$ of $G$. We define the corresponding Galois semi-covering functor as $F^\lambda: \mathrm{mod}\mbox{-}\bar\Lambda \to \mathrm{mod}\mbox{-}\Lambda$. 
\end{section}


\section{Galois semi-covering functor over morphism category}\label{MGC}
In this section, we present a Galois semi-covering functor between the (mono)morphism categories of an algebra and its skew group algebra. Recall that for an abelian category $A$, one can define the morphism category $\mathrm{H}(A)$, also an abelian category as follows. The objects of $\mathrm{H}(A)$ are all morphisms $(X\xrightarrow{f} Y)$ in $A$. The morphisms from object $(X\xrightarrow{f} Y)$  to $(X'\xrightarrow{f'} Y')$ are pairs $(a,b)$, where $a:X\to X'$ and $b:Y\to Y'$ such that $b\circ f = f'\circ a$. The composition of morphisms is component-wise. The monomorphism category of $A$, denoted $\mathrm{S}(A)$, is the full subcategory of $\mathrm{H}(A)$ consisting of all monomorphisms in $A$. Likewise $F_\lambda: \mathrm{mod}\mbox{-}\Lambda \to \mathrm{mod}\mbox{-}\bar\Lambda$, we show that the induced functors $\mathrm{H}F_\lambda: \mathrm{H}(\mathrm{mod}\mbox{-}\Lambda)\to \mathrm{H}(\mathrm{mod}\mbox{-}\bar{\Lambda})$ and $\mathrm{S}F_\lambda: \mathrm{S}(\mathrm{mod}\mbox{-}\Lambda)\to \mathrm{S}(\mathrm{mod}\mbox{-}\bar{\Lambda})$ are also Galois semi-covering.

\vspace{0.1in}
\noindent{\textbf{G-action over $\mathrm{H}(\mathrm{mod}\mbox{-}\Lambda)$}}: Let $g\in G$ and $\mathbb{X}=(X\xrightarrow{f}Y)$ be an object in $\mathrm{H}(\mathrm{mod}\mbox{-}\Lambda)$. We denote by ${}^g\mathbb{X}=({}^gX\xrightarrow{{}^gf} {}^gY)$ the object in $\mathrm{H}(\mathrm{mod}\mbox{-}\Lambda)$ by acting $g$ component-wise on $\mathbb{X}$. Assume a morphism $\psi=(\psi_1, \psi_2):\mathbb{X}\to \mathbb{Y}$ in $\mathrm{H}(\mathrm{mod}\mbox{-}\Lambda)$ is given. Define ${}^g\psi=({}^g\psi_1, {}^g\psi_2)$ a morphism in $\mathrm{H}(\mathrm{mod}\mbox{-}\Lambda)$ from ${}^g\mathbb{X}$ to ${}^g\mathbb{Y}$ obtained by acting $g$ component-wise on $\psi$. We observe that $g$ induces an automorphism on $\mathrm{H}(\mathrm{mod}\mbox{-}\Lambda)$. Denote again by $g$ the automorphism on $\mathrm{H}(\mathrm{mod}\mbox{-}\Lambda)$ induced by the automorphism $g:\Lambda \to \Lambda$. The set of all the induced automorphisms $g: \mathrm{H}(\mathrm{mod}\mbox{-}\Lambda) \to \mathrm{H}(\mathrm{mod}\mbox{-}\Lambda)$ makes a group of automorphisms on $\mathrm{H}(\mathrm{mod}\mbox{-}\Lambda)$, which is isomorphic to $G$. We identify the newly obtained group by $G$. This define an action of $G$ on $\mathrm{H}(\mathrm{mod}\mbox{-}\Lambda)$. Since $g: \Lambda \to \Lambda$ is an automorphism, $\psi$ is a monomorphism if and only if so is ${}^g\psi$. Thus, the action of $G$ on $\mathrm{H}(\mathrm{mod}\mbox{-}\Lambda)$ is restricted on $\mathrm{S}(\mathrm{mod}\mbox{-}\Lambda)$. 

\vspace{0.1in}
\noindent{\textbf{Induced functor $\mathrm{H}F_\lambda$ (resp. $\mathrm{S}F_\lambda$) over $\mathrm{H}(\mathrm{mod}\mbox{-}\Lambda)$ (rep. $\mathrm{S}(\mathrm{mod}\mbox{-}\Lambda)$)}}: Using the same idea (applying component-wisely a functor), we can extend the functor $F_\lambda: \mathrm{mod}\mbox{-}\Lambda\to \mathrm{mod}\mbox{-}\bar{\Lambda}$ to the functor $HF_\lambda: \mathrm{H}(\mathrm{mod}\mbox{-}\Lambda) \to \mathrm{H}(\mathrm{mod}\mbox{-}\bar{\Lambda})$, sending $(X\xrightarrow{f} Y) \in \mathrm{H}(\mathrm{mod}\mbox{-}\Lambda)$ to $(F_\lambda(X)\xrightarrow{F_\lambda(f)} F_\lambda(Y))$ in $\mathrm{H}(\mathrm{mod}\mbox{-}\bar{\Lambda})$. Since $F_\lambda$ is exact, we can restrict $\mathrm{H}F_\lambda$ to the corresponding monomorphism categories, denoted by $\mathrm{S}F_\lambda: \mathrm{S}(\mathrm{mod}\mbox{-}\Lambda)\to \mathrm{S} (\mathrm{mod}\mbox{-}\bar{\Lambda})$. Analogously, the functor $\mathrm{H}F^\lambda: \mathrm{H}(\mathrm{mod}\mbox{-}\bar{\Lambda})\to \mathrm{H}(\mathrm{mod}\mbox{-}\Lambda)$ is obtained by applying component-wise the functor $F^\lambda: \mathrm{mod}\mbox{-}\bar{\Lambda}\to \mathrm{mod}\mbox{-}\Lambda$. Since the functor $F^\lambda$ is exact, we also get the restricted functor $\mathrm{S}F^\lambda : \mathrm{S}(\mathrm{mod}\mbox{-}\bar{\Lambda})\to \mathrm{S}(\mathrm{mod}\mbox{-}\Lambda$).

For the sake of simplicity, we sometimes denote $\mathbb{X}= (X\xrightarrow{f} Y)\in \mathrm{H}(\mathrm{mod}\mbox{-}\Lambda)$ as $f_{XY}$ (or, simply $f$ if the source and target is clear). The following Proposition \cite{GoSaTr25} show that, unlike Galois covering, Galois semi-covering does not necessarily preserve irreducible morphisms.

\begin{proposition}\label{3.7}
Assume that $f_{MN} \in \mathrm{Hom}_\Lambda(M,N)$ for some $M,N \in \mathrm{mod}\mbox{-}\Lambda$. Then there exists a $\bar{f}$ with $\bar{f}_{{}^{\hat{g}}\bar{M}{}^{\hat{g}}\bar{N}} \in \mathrm{Hom}_{\bar\Lambda}({}^{\hat{g}}\bar{M},{}^{\hat{g}}\bar{N})$ such that the following hold: 

$$\mathrm{H}F_\lambda(f_{{}^gM,{}^gN})\approx \begin{cases}\uwithtext{$\bar{N}$}_{\hat{g}\in \hat{G}} \bar{f}_{{}^{\hat{g}}\bar{M} \bar{N}} &\mbox{ if } G_{M}= G, G_{N}\neq G;\\\uwithtext{$\bar{M}$}_{\hat{g}\in \hat{G}} \bar{f}_{\bar{M} {}^{\hat{g}}\bar{N}} &\mbox{ if }G_{M}\neq G, G_{N}= G;\\\bigoplus_{\hat{g}\in \hat{G}} \bar f_{{}^{\hat{g}}\bar{M}{}^{\hat{g}}\bar{N}} &\mbox{ if } G_{MN}=G;\\\bar{f}_{\bar{M}\bar{N}}&\mbox{ if }G_{M}, G_{N}\neq G.\end{cases}$$

Where, $\uwithtext{Z}_{i\in I}f_i$ means $f_i$'s for an index set $I$ and $i\in I$ being glued at the module $Z$.
\end{proposition}

The next result presents the adjoint functor of the functor $F_\lambda$. 

\begin{proposition}\label{gcad}
The Galois semi-covering functor $F^\lambda: \mathrm{mod}\mbox{-}\bar\Lambda \to \mathrm{mod}\mbox{-}\Lambda$ is the adjoint functor of the functor $F_\lambda: \mathrm{mod}\mbox{-}\Lambda \to \mathrm{mod}\mbox{-}\bar\Lambda$.
\end{proposition}

\begin{proof}
We show that for any $C\in \mathrm{mod}\mbox{-}\Lambda$ and $D\in \mathrm{mod}\mbox{-}\bar\Lambda$, there is an isomorphism of vector spaces $\mathrm{Hom}_\Lambda(F^\lambda D,C) \approx \mathrm{Hom}_{\bar\Lambda}(D,F_\lambda C)$. Following are the possible cases depending whether $D$ and $C$ are stable under $G$-action:
\begin{itemize}
    \item[$G_D, G_C=G$] We have, $\mathrm{Hom}_\Lambda(F^\lambda D,C)=\mathrm{Hom}_\Lambda(\bigoplus_{g\in G} {}^g \bar{D},C)=\bigoplus_{g\in G}\mathrm{Hom}_\Lambda({}^g\bar{D},C)$ by Remark \ref{indsk}. Also, $\mathrm{Hom}_{\bar\Lambda}(D, F_\lambda C)=\mathrm{Hom}_{\bar\Lambda}(D, \bigoplus_{\hat{g}\in \hat{G}} {}^{\hat{g}}\bar{C})=\bigoplus_{\hat{g}\in \hat{G}}\mathrm{Hom}_{\bar\Lambda}(D,{}^{\hat{g}}\bar{C})$. Then the result follows by Theorem \ref{GCM} and Proposition \ref{3.7}.

    \item[$G_D\neq G, G_C\neq G$] We have, $\mathrm{Hom}_\Lambda(F^\lambda D,C)=\mathrm{Hom}_\Lambda(\bar{D},C)$ by Remark \ref{indsk}. Moreover, $\mathrm{Hom}_{\bar\Lambda}(D, F_\lambda C)=\mathrm{Hom}_{\bar\Lambda}(D, \bar{C})$. Then the result follows by Theorem \ref{GCM} and Proposition \ref{3.7}.
    
    \item[$G_D\neq G, G_C=G$] We have, $\mathrm{Hom}_\Lambda(F^\lambda D,C)=\mathrm{Hom}_\Lambda(\bar{D},C)$ by Remark \ref{indsk}. Moreover, $\mathrm{Hom}_{\bar\Lambda}(D, F_\lambda C)=\mathrm{Hom}_{\bar\Lambda}(D, \bigoplus_{\hat{g}\in \hat{G}} {}^{\hat{g}}\bar{C})=\bigoplus_{\hat{g}\in \hat{G}}\mathrm{Hom}_{\bar\Lambda}(D,{}^{\hat{g}}\bar{C})$. Now, by Theorem \ref{GCM}, we have $\mathrm{Hom}_{\bar\Lambda}(F_\lambda \bar{D}, F_\lambda C)=\mathrm{Hom}_\Lambda^{|G|}(\bar{D}, C)= \mathrm{Hom}_\Lambda^{|G|}(F^\lambda D,C)$. Again using Remark \ref{indsk}, we have $\mathrm{Hom}_{\bar\Lambda}(F_\lambda \bar{D}, F_\lambda C)=\mathrm{Hom}_{\bar\Lambda}(\bigoplus_{g\in G} {}^gD, F_\lambda C)=\bigoplus_{g\in G} \mathrm{Hom}_{\bar\Lambda}({}^gD, F_\lambda C)$. Since, $\mathrm{Hom}_{\bar\Lambda}({}^{g_1}D, F_\lambda C)\approx \mathrm{Hom}_{\bar\Lambda}({}^{g_2}D, F_\lambda C)$ for any $g_1, g_2\in G$ then it is clear that $\mathrm{Hom}_{\bar\Lambda}(D, F_\lambda C)\approx \mathrm{Hom}_\Lambda(F^\lambda D,C)$.

    \item[$G_D=G, G_C\neq G$] The case is similar to the previous case.
\end{itemize}
Hence the result. \end{proof}

\begin{rmk}
Since $(F_\lambda, F^\lambda)$ is an adjoint pair, for any $X\in \mathrm{mod}\mbox{-}\Lambda$ and $M\in \mathrm{mod}\mbox{-}\bar{\Lambda}$ there are two natural isomorphisms $\eta_{X,M}: \mathrm{Hom}_{\bar{\Lambda}}(F_\lambda(X),M)\to \mathrm{Hom}_\Lambda(X, F^\lambda(M))$ and its inverse $\eta^{-1}_{X,M}: \mathrm{Hom}_\Lambda(X, F^\lambda(M))\to \mathrm{Hom}_{\bar{\Lambda}}(F_\lambda(X),M)$.
\end{rmk}
Next we show that the functors $\mathrm{H}F_\lambda$ and $\mathrm{H}F^\lambda$ play similar role as $F^\lambda$ and $F_\lambda$ respectively.

\begin{proposition}\label{GSTAB}
Suppose $\mathbb{X}, \mathbb{X}_1, \mathbb{X}_2\in \mathrm{H}(\mathrm{mod}\mbox{-}\Lambda)$. Then the following holds.
\begin{enumerate}
\item $\mathrm{H}F_\lambda({}^g\mathbb{X})\approx \mathrm{H}F_\lambda(\mathbb{X})$, for any $g\in G$.
\item $\mathrm{H}F^\lambda \circ \mathrm{H}F_\lambda(\mathbb{X})\approx \bigoplus_{g\in G} {}^g\mathbb{X}$.
\item For indecomposables $\mathbb{X}_1$ and $\mathbb{X}_2$, $\mathrm{H}F_\lambda(\mathbb{X}_1)\approx \mathrm{H}F_\lambda(\mathbb{X}_2)$ implies $\mathbb{X}_1\approx {}^g \mathbb{X}_2$, for some $g\in G$.
\end{enumerate}
\end{proposition}
\begin{proof}
$(1)$ It follows from the proposition \ref{3.7}.

$(2)$ Given $\mathbb{X}: A\to B$, consider the following isomorphisms
$\mathrm{H}F^\lambda \circ \mathrm{H}F_\lambda(\mathbb{X})\approx (F^\lambda \circ F_\lambda(A)\xrightarrow{F^\lambda \circ F_\lambda(f)} F^\lambda \circ F_\lambda(B))\approx (\bigoplus_{g\in G}{}^gA\xrightarrow{\bigoplus_{g\in G}{}^gf}  \bigoplus_{g\in G}{}^gB) \approx \bigoplus_{g\in G}({}^gA\xrightarrow{{}^gf} {}^gB)\approx \bigoplus_{g\in G}{}^g\mathbb{X}$. The second isomorphism follows from the canonical isomorphism $F^\lambda \circ F_\lambda(C)\approx \bigoplus_{g\in G}{}^gC$, for any $C \in A$. 


$(3)$ Assume $\mathrm{H}F_\lambda(\mathbb{X}_1)\approx \mathrm{H}F_\lambda(\mathbb{X}_2)$. Applying $\mathrm{H}F^\lambda$, from part (2) it follows $\bigoplus_{g\in G}{}^g\mathbb{X}_1\approx \mathrm{H}F^\lambda \circ \mathrm{H}F_\lambda(\mathbb{X}_1)\approx \mathrm{H}F^\lambda \circ \mathrm{H}F_\lambda(\mathbb{X}_2)\approx \bigoplus_{g\in G}{}^g\mathbb{X}_2$. Hence $\mathbb{X}_1$ is a summand of $\bigoplus_{g\in G}{}^g\mathbb{X}_2$. The Krull-Schmidt property of $\mathrm{H}(\mathrm{mod}\mbox{-}\Lambda)$ implies the existence of $g\in G$ such that $\mathbb{X}_1\approx {}^g\mathbb{X}_2$, as desired. 
\end{proof}

Since the functors $\mathrm{S}F_\lambda$ and $\mathrm{S}F^\lambda$ are restrictions of the functors $\mathrm{H}F_\lambda$ and $\mathrm{H}F^\lambda$ over the monomorphism categories, they also satisfy similar properties as described in the above proposition. 

Below we show that the functors $\mathrm{H}F_\lambda$ and $\mathrm{H}F^\lambda$ (resp. $\mathrm{S}F_\lambda$ and $\mathrm{S}F^\lambda$) are adjoint to each other.

\begin{lemma}\label{HADJ}
Let $\mathbb{X}$ be in $\mathrm{H}(\mathrm{mod}\mbox{-}\Lambda)$ and $\mathbb{M}$ in $\mathrm{H}(\mathrm{mod}\mbox{-}\bar{\Lambda})$. Then there exists the following natural isomorphisms of $\mathcal{K}$-vector spaces $\zeta_{\mathbb{X},\mathbb{M}}: \mathrm{H}(\mathrm{mod}\mbox{-}\bar{\Lambda})(\mathrm{H}F_\lambda(\mathbb{X}), \mathbb{M})\to \mathrm{H}(\mathrm{mod}\mbox{-}\Lambda)(\mathbb{X}, \mathrm{H}F^\lambda(\mathbb{M}))$.
In particular, similar isomorphism occurs for the functors $\mathrm{S}F_\lambda$ and $\mathrm{S}F^\lambda$ when $\mathbb{X} \in \mathrm{S} (\mathrm{mod}\mbox{-}\Lambda)$ and $\mathbb{M}\in \mathrm{S} (\mathrm{mod}\mbox{-}\bar{\Lambda})$.    
\end{lemma}
\begin{proof}
Assume $\mathbb{X}= (X\xrightarrow{f} Y)$ and $\mathbb{M}= (M \xrightarrow{h} N)$. Based on the adjoint isomorphism of the adjoint pair $(F_\lambda, F^\lambda)$, introduced in Proposition \ref{gcad}, we define the following $K$-linear map $$\zeta_{\mathbb{X},\mathbb{M}}: \mathrm{H}(\mathrm{mod}\mbox{-}\bar{\Lambda})(\mathrm{H}F_\lambda(\mathbb{X}), \mathbb{M})\to \mathrm{H}(\mathrm{mod}\mbox{-}\Lambda)(\mathbb{X}, \mathrm{H}F^\lambda(\mathbb{M}))$$
by sending $\zeta= (\zeta_1, \zeta_2)$ in $\mathrm{H}(\mathrm{mod}\mbox{-}\bar{\Lambda})(\mathrm{H}F_\lambda(\mathbb{X}), \mathbb{M})$ to
$\eta = (\eta_{\mathbb{X},\mathbb{M}}(\zeta_1), \eta_{\mathbb{Y},\mathbb{N}}(\zeta_2))$ in $\mathrm{H}(\mathrm{mod}\mbox{-}\Lambda)(\mathbb{X}, \mathrm{H}F^\lambda(\mathbb{M}))$. By a direct computation and following the definitions of $\eta_{-,-}, F^\lambda$ and $F_\lambda$, one can show that $\eta$ is a morphism in $\mathrm{H}(\mathrm{mod}\mbox{-}\bar{\Lambda})$. Similarly, one can define the $K$-linear map $$\zeta'_{\mathbb{X},\mathbb{M}}: \mathrm{H}(\mathrm{mod}\mbox{-}\Lambda)(\mathbb{X}, \mathrm{H}F^\lambda(\mathbb{M}))\to \mathrm{H}(\mathrm{mod}\mbox{-}\bar{\Lambda})(\mathrm{H}F_\lambda(\mathbb{X}), \mathbb{M})$$ by sending $\eta = (\eta_1, \eta_2)$ in $\mathrm{H}(\mathrm{mod}\mbox{-}\Lambda)(\mathbb{X}, \mathrm{H}F^\lambda(\mathbb{M}))$ to $\zeta= (\eta^{-1}_{\mathbb{X},\mathbb{M}}(\eta_1), \eta^{-1}_{\mathbb{Y},\mathbb{N}}(\eta_1))$ in $\mathrm{H}(\mathrm{mod}\mbox{-}\bar{\Lambda})(\mathrm{H}F_\lambda(\mathbb{X}), \mathbb{M})$. It is easy to see that $\zeta_{\mathbb{X},\mathbb{M}}$ and $\zeta'_{\mathbb{X},\mathbb{M}}$ are mutually inverse. Hence $\zeta_{\mathbb{X},\mathbb{M}}$ is an isomorphism.
\end{proof}

The following theorem shows the existence of the Galois semi-covering between the homomorphism categories of an algebra and its skew group algebra.

\begin{theorem}\label{HGCM}
Assume that a group $G$ acts on an algebra $\Lambda$ and $\bar\Lambda$ is the associated skew group algebra. Then for any $f, h\in \mathrm{H}(\mathrm{mod}\mbox{-}\Lambda)$, the functor $\mathrm{H}F_\lambda: \mathrm{H}(\mathrm{mod}\mbox{-}\Lambda) \to \mathrm{H}(\mathrm{mod}\mbox{-}\bar\Lambda)$ induces the following isomorphisms of vector spaces:
$$\mathrm{H}(\mathrm{mod}\mbox{-}\bar\Lambda)(\mathrm{H}F_\lambda f,\mathrm{H}F_\lambda h)\approx \begin{cases}\bigoplus_{g\in G} \mathrm{H}(\mathrm{mod}\mbox{-}\Lambda)({}^gf, h)&\mbox{ if } G_{f}\neq G;\\\bigoplus_{g\in G} \mathrm{H}(\mathrm{mod}\mbox{-}\Lambda)(f, {}^gh)&\mbox{ if }G_{h}\neq G;\\\mathrm{H}(\mathrm{mod}\mbox{-}\Lambda)^{|G|}(f, h)&\mbox{ if }G_{fh}= G.\end{cases}$$  
\end{theorem}

\begin{proof}
Here, we analyze the first two cases, taking into account three different situations where $G_{f}\neq G, G_{h}= G$; $G_{h}\neq G, G_{f}= G$ and $G_{f}\neq G, G_{h}\neq G$. Now, let us describe the isomorphism $\nu_{f,h}$ explicitly as follows:

\noindent{\textbf{Case-I:}} Here, $G_{f}\neq G, G_{h}= G$. Thus $F_\lambda ({}^gf)=\bar {f}=\begin{cases}\uwithtext{$\bar{N}$}_{\hat{g}\in \hat{G}} \bar{f}_{{}^{\hat{g}}\bar{M} \bar{N}} &\mbox{ if } G_{M}= G, G_{N}\neq G;\\\uwithtext{$\bar{M}$}_{\hat{g}\in \hat{G}} \bar{f}_{\bar{M} {}^{\hat{g}}\bar{N}} &\mbox{ if }G_{M}\neq G, G_{N}= G;\\\bar{f}_{\bar{M}\bar{N}}&\mbox{ if }G_{M}, G_{N}\neq G.\end{cases}$ and $F_\lambda ({}^g h)=\bigoplus_{\hat{g}\in \hat{G}} \bar h_{{}^{\hat{g}}\bar{M}{}^{\hat{g}}\bar{N}}$ for all $g\in G$ and $\bar f, \bar h \in \mathrm{H}(\mathrm{mod}\mbox{-}\bar\Lambda)$ by Proposition \ref{3.7}. We define $\mathrm{H}\nu_{f, h}: \bigoplus_{g\in G} \mathrm{H}(\mathrm{mod}\mbox{-}\Lambda)({}^gf, h) \to \mathrm{H}(\mathrm{mod}\mbox{-}\bar\Lambda)(\mathrm{H}F_\lambda f,\mathrm{H}F_\lambda h)$ by $\mathrm{H}\nu_{M,N}(F_1,\hdots, F_n)_{1\times n}=(\bar{F_1},\hdots, \bar{F_n})_{n\times 1}$, and the inverse isomorphism $\mathrm{H}\mu_{M, N}$ by $\mathrm{H}\mu_{f, h}(\bar{F_1},\hdots, \bar{F_n})_{n\times 1}=(F_1,\hdots, F_n)_{1\times n}$, where, $F_i=F_{{}^{g_i}fh}: {}^{g_i} f\to h$ and $\bar{F_i}=\bar{F}_{\bar{f}{}^{\hat{g}_i}\bar{h}}: \bar{f}\to {}^{\hat{g}_i}\bar{h}$ are morphisms in $\mathrm{H}(\mathrm{mod}\mbox{-}\Lambda)$ and $\mathrm{H}(\mathrm{mod}\mbox{-}\bar\Lambda)$ respectively. Clearly, $\mathrm{H}\nu_{f,h}$ and $\mathrm{H}\mu_{f,h}$ are linear morphisms.
\vspace{0.1in}
\begin{figure}[h]
\centering
\begin{tikzcd} [sep={2.7em,between origins}]
{}^{g_1}f \arrow[rrdd, "F_1"]                             &  &   &                             &                              &                                                                                                                     &  & {}^{\hat{g}_1} \bar{h}                             \\
{}^{g_2}f \arrow[rrd, "F_2"'] \arrow[dd, no head, dotted] &  &   & {}                          & {} \arrow[l, "{\mathrm{H}\mu_{f,h}}"'] &                                                                                                                     &  & {}^{\hat{g}_2} \bar{h} \arrow[dd, no head, dotted] \\
                                                          &  & h &                             &                              & \bar{f} \arrow[rruu, "\bar{F}_1"] \arrow[rru, "\bar{F}_2"'] \arrow[rrdd, "\bar{F}_n"'] \arrow[rrd, "\bar{F}_{n-1}"] &  &                                                    \\
{}^{g_{n-1}}f \arrow[rru, "F_{n-1}"]                      &  &   & {} \arrow[r, "{\mathrm{H}\nu_{f,h}}"] & {}                           &                                                                                                                     &  & {}^{\hat{g}_{n-1}} \bar{h}                         \\
{}^{g_n}f \arrow[rruu, "F_n"']                            &  &   &                             &                              &                                                                                                                     &  & {}^{\hat{g}_n} \bar{h}                            
\end{tikzcd}
\end{figure}

\noindent{\textbf{Case-II:}} Here, $G_{h}\neq G, G_{f}= G$. Thus $F_\lambda ({}^gh)=\bar h=\begin{cases}\uwithtext{$\bar{N}$}_{\hat{g}\in \hat{G}} \bar{f}_{{}^{\hat{g}}\bar{M} \bar{N}} &\mbox{ if } G_{M}= G, G_{N}\neq G;\\\uwithtext{$\bar{M}$}_{\hat{g}\in \hat{G}} \bar{f}_{\bar{M} {}^{\hat{g}}\bar{N}} &\mbox{ if }G_{M}\neq G, G_{N}= G;\\\bar{f}_{\bar{M}\bar{N}}&\mbox{ if }G_{M}, G_{N}\neq G.\end{cases}$ and $F_\lambda (f)=\bigoplus_{\hat{g}\in \hat{G}} {}^{\hat{g}} \bar f$ for all $g\in G$ and $\bar f, \bar h \in \mathrm{H}(\mathrm{mod}\mbox{-}\bar\Lambda)$ by Proposition \ref{3.7}. We define $\mathrm{H}\nu_{f, h}: \bigoplus_{g\in G} \mathrm{H}(\mathrm{mod}\mbox{-}\Lambda)(f, {}^gh) \to \mathrm{H}(\mathrm{mod}\mbox{-}\bar\Lambda)(\mathrm{H}F_\lambda f,\mathrm{H}F_\lambda h)$ by $\mathrm{H}\nu_{f,h}(F_1,\hdots, F_n)_{1\times n}=(\bar{F_1},\hdots, \bar{F_n})_{n\times 1}$, and the inverse isomorphism $\mathrm{H}\mu_{M, N}$ by $\mathrm{H}\mu_{f, h}(\bar{F_1},\hdots, \bar{F_n})_{n\times 1}=(F_1,\hdots, F_n)_{1\times n}$, where, $F_i=F_{f{}^{g_i}h}: f\to {}^{g_i}h$ and $\bar{F_i}=\bar{F}_{{}^{\hat{g}_i}\bar{f}\bar{h}}: {}^{\hat{g}_i}\bar{f}\to \bar{h}$ are morphisms in $\mathrm{H}(\mathrm{mod}\mbox{-}\Lambda)$ and $\mathrm{H}(\mathrm{mod}\mbox{-}\bar\Lambda)$ respectively. Clearly, $\mathrm{H}\nu_{M,N}$ and $\mathrm{H}\mu_{M,N}$ are linear morphisms.

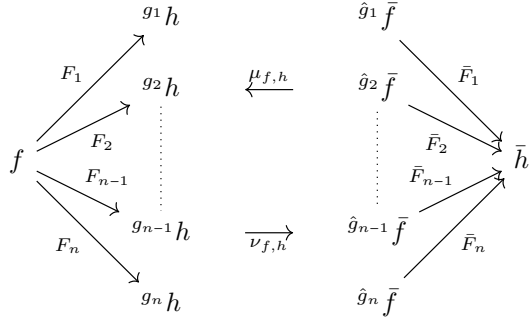
\[\begin{tikzcd}[sep={2.7em,between origins}]
                                                                                      &  & {}^{g_1}h                             &                              &                              & {}^{\hat{g}_1}\bar{f} \arrow[rrdd, "\bar{F}_1"]                             &  &         \\
                                                                                      &  & {}^{g_2}h \arrow[dd, no head, dotted] & {}                           & {} \arrow[l, "{\mu_{f,h}}"'] & {}^{\hat{g}_2}\bar{f} \arrow[rrd, "\bar{F}_2"'] \arrow[dd, no head, dotted] &  &         \\
f \arrow[rruu, "F_1"] \arrow[rru, "F_2"'] \arrow[rrdd, "F_n"'] \arrow[rrd, "F_{n-1}"] &  &                                       &                              &                              &                                                                             &  & \bar{h} \\
                                                                                      &  & {}^{g_{n-1}}h                         & {} \arrow[r, "{\nu_{f,h}}"'] & {}                           & {}^{\hat{g}_{n-1}}\bar{f} \arrow[rru, "\bar{F}_{n-1}"]                      &  &         \\
                                                                                      &  & {}^{g_n}h                             &                              &                              & {}^{\hat{g}_n}\bar{f} \arrow[rruu, "\bar{F}_n"']                            &  &        
\end{tikzcd}\]

\noindent{\textbf{Case-III:}} Here, $G_f\neq G, G_h\neq G$. Then it works as a locally Galois covering as $G$-action does not fix $f$ and $h$ both, and as a result, the number of arrows from ${}^gf$ to $h$ for $g\in G$ is equal to the number of arrows from $\bar{f}$ to $\bar{h}$ (= the number of arrows from $f$ to ${}^gh$ also) i.e. $$\mathrm{dim}_K\mathrm{H}(\mathrm{mod}\mbox{-}\bar\Lambda)(\bar{f}, \bar{h})= \mathrm{dim}_{K}\mathrm{H}(\mathrm{mod}\mbox{-}\Lambda)(\bigoplus_{g\in G}{}^gf, h)=\mathrm{dim}_{K}\bigoplus_{g\in G}\mathrm{H}(\mathrm{mod}\mbox{-}\Lambda)({}^gf, h).$$

\[
\begin{tikzcd}[sep={2.85em,between origins}]
{}^{g_1}f \arrow[rrdd, "F_1"]                             &  &                              & {} \arrow[rr, "\bar{F}_1"]                           &                                & {}                          &                                                                                       &  & {}^{g_1}h                             \\
{}^{g_2}f \arrow[rrd, "F_2"'] \arrow[dd, no head, dotted] &  & {}                           & {} \arrow[l, "{\mu_{f,h}}"'] \arrow[rr, "\bar{F}_2"] & {} \arrow[dd, no head, dotted] & {} \arrow[r, "{\mu_{f,h}}"] & {}                                                                                    &  & {}^{g_2}h \arrow[dd, no head, dotted] \\
                                                          &  & h                            & \bar{f}                                              &                                & \bar{h}                     & f \arrow[rruu, "F^1"] \arrow[rru, "F^2"'] \arrow[rrd, "F^{n-1}"] \arrow[rrdd, "F^n"'] &  &                                       \\
{}^{g_{n-1}}f \arrow[rru, "F_{n-1}"]                      &  & {} \arrow[r, "{\nu_{f,h}}"'] & {} \arrow[rr, "\bar{F}_{n-1}"]                       & {}                             & {}                          & {} \arrow[l, "{\nu_{f,h}}"]                                                           &  & {}^{g_{n-1}}h                         \\
{}^{g_n}f \arrow[rruu, "F_n"']                            &  &                              & {} \arrow[rr, "\bar{F}_n"]                           &                                & {}                          &                                                                                       &  & {}^{g_n}h                            
\end{tikzcd}
\]

The second equality holds as it is a finite-dimensional vector space. Thus we get $$\mathrm{H}(\mathrm{mod}\mbox{-}\bar\Lambda)(\bar{f}, \bar{h})\approx\bigoplus_{g\in G}\mathrm{H}(\mathrm{mod}\mbox{-}\Lambda)({}^gf, h)\approx\bigoplus_{g\in G}\mathrm{H}(\mathrm{mod}\mbox{-}\Lambda)(f, {}^gh).$$ 

Let $m\leq n$ be the number of the non-zero representatives $f_i$ from the equivalence classes $\mathcal{O}_{f_i}$ under the $G$-action. Then the isomorphism $\mathrm{H}\nu_{f, h}: \bigoplus_{g\in G} \mathrm{H}(\mathrm{mod}\mbox{-}\Lambda)({}^gf,h) \to \mathrm{H}(\mathrm{mod}\mbox{-}\bar\Lambda)(\bar{f}, \bar{h})$ is given by $\mathrm{H}\nu_{f,h}(F_1,\hdots, F_m)_{1\times m}=(\bar{F_1},\hdots, \bar{F_m})$ and the inverse $\mathrm{H}\mu_{f, h}$ by $\mu_{f, h}(\bar{F_1},\hdots, \bar{F_m})=(F_1,\hdots, F_m)_{1\times m}$ where, $F_i=F_{{}^{g_i}fh}: {}^{g_i} f\to h$ and $\bar{F_i}=\bar{F}_{\bar{f}\bar{h}}: \bar{f}\to \bar{h}$ are morphisms in $\mathrm{H}(\mathrm{mod}\mbox{-}\Lambda)$ and $\mathrm{H}(\mathrm{mod}\mbox{-}\bar\Lambda)$ respectively. Moreover, the second isomorphism in the above expression is defined similarly.

Note that, we have the following equality: $$\mathrm{dim}_K\mathrm{H}^{|\tilde{G}|}(\mathrm{mod}\mbox{-}\bar\Lambda)(\bar{f}, \bar{h})= \mathrm{dim}_K\mathrm{H}(\mathrm{mod}\mbox{-}\Lambda)(\bigoplus_{g\in G}{}^gf, \bigoplus_{g\in G}{}^gh).$$

\noindent{\textbf{Case-IV:}} Here, $G_{fh}= G$. Hence, $F_\lambda (f)=\bigoplus_{\hat{g}\in \hat{G}} {}^{\hat{g}} \bar f$ and $F_\lambda (h)=\bigoplus_{\hat{g}\in \hat{G}} {}^{\hat{g}} \bar h$ for some $\bar f, \bar h \in \mathrm{H}(\mathrm{mod}\mbox{-}\bar\Lambda)$ by Proposition \ref{3.7}. From the above discussion, it is clear that, $$\mathrm{dim}_K\mathrm{H}^{|G|}(\mathrm{mod}\mbox{-}\Lambda)(f, h)= \mathrm{dim}_K \mathrm{H}(\mathrm{mod}\mbox{-}\bar{\Lambda})(\bigoplus_{\hat{g}\in \hat{G}}{}^{\hat{g}}\bar{f}, \bigoplus_{\hat{g}\in \hat{G}}{}^{\hat{g}}\bar{h})= \mathrm{dim}_K\mathrm{H}(\mathrm{mod}\mbox{-}\bar{\Lambda})(F_\lambda (f), F_\lambda (h)).$$ 

Therefore, we have, $\mathrm{H}(\mathrm{mod}\mbox{-}\bar\Lambda)(F_\lambda (f), F_\lambda (h))=\mathrm{H}^{|G|}(\mathrm{mod}\mbox{-}\Lambda)(f, h)$, since it is a finite-dimensional vector space. Here, the explicit isomorphism $\mathrm{H}\nu_{f,h}: \mathrm{H}^{|G|}(\Lambda)(f, h)\to \mathrm{H}(\bar{\Lambda})(F_\lambda (f), F_\lambda (h))$ is given by sending $F_i\mapsto \bigoplus_{g\in G} {}^g\bar{F_i}$. One can easily verify that it is a monomorphism; hence, the isomorphism follows, since both vector spaces have the same dimension.
\end{proof}

We conclude this section with an example to illustrate the concept of semi-covering.
\begin{example}\label{GCMOR}
Consider the algebra $\Lambda$ and its skew algebra $\bar{\Lambda}$ in example \cite[Example~3.10]{GoSaTr25}.

\noindent{\textbf{Case-I:}} Assume that $G_{f}\neq G, G_{h}= G$. Here, we analyze the first two cases, taking into account three different situations where $G_{f_s}\neq G, G_{f_t}\neq G$; $G_{f_s}\neq G, G_{f_t}= G$ and $G_{f_s}= G, G_{f_t}\neq G$. 

\underline{$G_{f_s}\neq G, G_{f_t}\neq G$}: Consider $f:\nd{3\ 1\\2\ 4} \to \nd{2\\ 3\ 1\\2\ 4}$ and $h:\nd{3\ 1\ 5\\2\ 4} \to \nd{2\ 4\\ 3\ 1\ 5\\2\ 4}$. Here, ${}^gf:\nd{1\ 5\\2\ 4} \to \nd{4\\ 1\ 5\\2\ 4}$ and ${}^gh=h$. Note that, $\mathrm{H}F_\lambda f:\nd{1\ 3\ 4\\2\ 2} \to \nd{2\\ 1\ 3\ 4\\2\ 2}$ and $\mathrm{H}F_\lambda h:=\bigoplus_{\hat{g}\in \mathbb{Z}_2}{}^{\hat{g}}\bar{h}$ where, $\bar{h}:\nd{1\ 4\\2} \to \nd{2\\ 1\ 4\\2}$ and ${}^{\hat{g}}\bar{h}:\nd{1\ 3\\2} \to \nd{2\\ 1\ 3\\2}$. Clearly, $\mathrm{dim}_K \mathrm{H}(\mathrm{mod}\mbox{-}\bar\Lambda)(\mathrm{H}F_\lambda f,\mathrm{H}F_\lambda h)=\mathrm{dim}_K\bigoplus_{g\in G} \mathrm{H}(\mathrm{mod}\mbox{-}\Lambda)(gf, h)=2$.

\underline{$G_{f_s}\neq G, G_{f_t}= G$}: Consider $f:\nd{2} \to \nd{1\ 4\\2}$ and $h:\nd{1\\2} \to \nd{2\\ 1\ 1\ 3\ 4\\ 2\ 2}$. Here, ${}^gf:\nd{ 2} \to \nd{ 1\ 3\\2}$ and ${}^gh=h$. Note that, $\mathrm{H}F_\lambda f=\uwithtext{$\bar{N}$}_{\hat{g}\in \mathbb{Z}_2}{}^{\hat{g}}\bar{f}:\nd{2}\bigoplus \nd{4} \to \nd{3\ 1\ 5\\ 2\ 4}$ where, $\bar{f}:\nd{2} \to \nd{3\ 1\ 5\\ 2\ 4}$, ${}^{\hat{g}}\bar{f}:\nd{4} \to \nd{3\ 1\ 5\\ 2\ 4}$ and $\bar{N}=\nd{3\ 1\ 5\\ 2\ 4}$, whereas $\mathrm{H}F_\lambda h=\bigoplus_{\hat{g}\in \mathbb{Z}_2}{}^{\hat{g}}\bar{h}$ where, $\bar{h}:\nd{3\\ 2} \to \nd{2\\ 3\ 1\ 5\\ 2\ 4}$ and ${}^{\hat{g}}\bar{h}:\nd{5\\ 4} \to \nd{4\\ 3\ 1\ 5\\ 2\ 4}$. Clearly, $\mathrm{dim}_K \mathrm{H}(\mathrm{mod}\mbox{-}\bar\Lambda)(\mathrm{H}F_\lambda f,\mathrm{H}F_\lambda h)=\mathrm{dim}_K\bigoplus_{g\in G} \mathrm{H}(\mathrm{mod}\mbox{-}\Lambda)({}^gf, h)=2$.

\vspace{0.05in}
\noindent{\textbf{Case-2:}} Assume that $G_{f}\neq G, G_{h}\neq G$. Here we consider the case where $G_{f_s}\neq G, G_{f_t}= G$ and $G_{h_s}\neq G, G_{h_t}= G$. Consider $f:\nd{2} \to \nd{1\\2\ 4}$ and $h:\nd{3\\2} \to \nd{3\ 1\ 5\\ 2\ 4}$. Here, ${}^gf:\nd{ 4} \to \nd{ 1\\2\ 4}$ and ${}^gh=\nd{5\\4} \to \nd{3\ 1\ 5\\ 2\ 4}$. Note that, $\mathrm{H}F_\lambda f=\uwithtext{$\bar{N}$}_{\hat{g}\in \mathbb{Z}_2}{}^{\hat{g}}\bar{f}:\nd{2} \to \nd{3\\ 2}\bigoplus \nd{4\\ 2}$ where, $\bar{f}:\nd{2} \to \nd{3\\ 2}$, ${}^{\hat{g}}\bar{f}:\nd{2} \to \nd{4\\ 2}$ and $\bar{N}=2$, whereas $\mathrm{H}F_\lambda h=\uwithtext{$\bar{N}$}_{\hat{g}\in \mathbb{Z}_2}{}^{\hat{g}}\bar{h}:\nd{1\\ 2} \to \nd{1\ 3\\ 2}\bigoplus \nd{1\ 4\\ 2}$ where, $\bar{h}:\nd{1\\ 2} \to \nd{1\ 3\\ 2}$, ${}^{\hat{g}}\bar{h}:\nd{1\\ 2} \to \nd{1\ 4\\ 2}$ and $\bar{N}= \nd{1\\ 2}$. Clearly, $\mathrm{dim}_K \mathrm{H}(\mathrm{mod}\mbox{-}\bar\Lambda)(\mathrm{H}F_\lambda f,\mathrm{H}F_\lambda h)=\mathrm{dim}_K\bigoplus_{g\in G} \mathrm{H}(\mathrm{mod}\mbox{-}\Lambda)({}^gf, h)=\mathrm{dim}_K\bigoplus_{g\in G} \mathrm{H}(\mathrm{mod}\mbox{-}\Lambda)(f, {}^gh)=1$.

\vspace{0.05in}
\noindent{\textbf{Case-3:}} Assume that $G_{f}= G, G_{h}= G$. Consider $f:\nd{2} \to \nd{2\\ 1\\2}$ and $h:\nd{1\\ 3\ 4\\2\ 2} \to \nd{2\\ 1\ 3\ 4\\2\ 2}$. Note that, $\mathrm{H}F_\lambda f:=\bigoplus_{\hat{g}\in \mathbb{Z}_2}{}^{\hat{g}}\bar{f}$  where, $\bar{f}:\nd{2} \to \nd{2\\ 3\\2}$ and ${}^{\hat{g}}\bar{f}:\nd{4} \to \nd{4\\ 5\\4}$, and $\mathrm{H}F_\lambda h:=\bigoplus_{\hat{g}\in \mathbb{Z}_2}{}^{\hat{g}}\bar{h}$ where, $\bar{h}:\nd{3\ 1\\2\ 4} \to \nd{2\\ 3\ 1\\2\ 5}$ and ${}^{\hat{g}}\bar{h}:\nd{1\ 5\\2\ 4} \to \nd{4\\ 1\ 5\\2\ 4}$. Clearly, $\mathrm{dim}_K \mathrm{H}(\mathrm{mod}\mbox{-}\bar\Lambda)(\mathrm{H}F_\lambda f,\mathrm{H}F_\lambda h)=\mathrm{dim}_K\mathrm{H}^{|G|}(\mathrm{mod}\mbox{-}\Lambda)(f, h)=2$.
\end{example}

\vspace{0.05in}
\begin{rmk}
Since the functor $\mathrm{S}F_\lambda$ is a restrictions of the functor $\mathrm{H}F_\lambda$ over the monomorphism category $\mathrm{S}(\mathrm{mod}\mbox{-}\Lambda)$, we can say that $\mathrm{S}F_\lambda: \mathrm{S}(\mathrm{mod}\mbox{-}\Lambda) \to \mathrm{S}(\mathrm{mod}\mbox{-}\bar\Lambda)$ is also a Galois semi-covering.
\end{rmk} 

\section{Galois semi-covering functor over functor categories}
In this section we show that the Galois semi-covering functor $\mathrm{H}F_\lambda$ between the morphism categories over an algebra and its skew group algebra induces a Galois semi-covering functor $\phi$ between their funtor categories. We also verify that the functor $\phi$ is exact and faithful.

\vspace{0.1in}
\noindent{\textbf{Serre subcategory:}} A full subcategory $S$ of $C$ is a Serre subcategory if and only if for any exact sequence $0 \to X \to Z \to Y \to 0$ in $S$ we have $Z \in S$ if and only if $X, Y \in S$. Therefore, $S$ is closed under subobjects, quotients and extensions. A Serre subcategory is an abelian subcategory. The kernel of an exact functor is a Serre subcategory of the domain category.

\vspace{0.1in}
\noindent{\textbf{Finitely presented functors:}} We denote by $\mathcal{G}(\Lambda)$ the category of all contravariant $K$-linear functors from $\mathrm{mod}\mbox{-}\Lambda$ to the category $\mathrm{mod}\mbox{-}K$ of finite dimensional $K$-vector spaces. Each $M\in \mathrm{mod}\mbox{-}\Lambda$ associates a contravariant hom-functor $\mathrm{H}_M \in \mathcal{G}(\Lambda)$ defined as $\mathrm{H}_M(X)= \mathrm{Hom}_\Lambda(X, M)$, for any $X \in \mathrm{mod}\mbox{-}\Lambda$, and is denoted by $\mathrm{Hom}_\Lambda(-, M)$. The homomorphism $f\in \mathrm{Hom}_\Lambda(M,N)$ corresponds to $\mathrm{Hom}_\Lambda(-, f): \mathrm{Hom}_\Lambda(-, M) \to \mathrm{Hom}_\Lambda(-,N)$ such that $\mathrm{Hom}_\Lambda(X, f): \mathrm{Hom}_\Lambda(X, M) \to \mathrm{Hom}_\Lambda(X, N)$ is defined by $\mathrm{Hom}_\Lambda(X,f)(g)= fg$, for any $g\in \mathrm{Hom}_\Lambda(X, M)$. By Yoneda lemma, the function $f\to \mathrm{Hom}_\Lambda(-,f)$ defines an isomorphism $\mathrm{Hom}_\Lambda(M, N) \to \mathcal{G}(\Lambda)(\mathrm{Hom}_\Lambda(-, M), \mathrm{Hom}_\Lambda(-, N))$ of vector spaces. Also, this yields $M \approx N$ if and only if $\mathrm{Hom}_\Lambda(-,M) \approx \mathrm{Hom}_\Lambda(-, N)$.

A functor $F \in \mathcal{G}(\Lambda)$ is finitely presented if and only if there exists an exact sequence of functors $\mathrm{Hom}_\Lambda(-, M) \xrightarrow{\mathrm{Hom}_\Lambda(-,f)} \mathrm{Hom}_\Lambda(-, N) \to F \to 0$, for some $M, N \in \mathrm{mod}\mbox{-}\Lambda$ and $\Lambda$-module homomorphism $f: M \to N$ i.e. $F \approx \mathrm{Coker}\mathrm{Hom}_\Lambda(-, f)$ which yields $F(X)$ is isomorphic to the cokernel of the map $\mathrm{Hom}_\Lambda(X,f): \mathrm{Hom}_\Lambda(X,M) \to \mathrm{Hom}_\Lambda(X, N)$. Denote $\mathcal{F}(\Lambda)$ the full subcategory of $\mathcal{G}(\Lambda)$ formed by finitely presented functors. Obviously $\mathrm{Hom}_\Lambda(-, M) \in \mathcal{F}(\Lambda)$. 


Additive subcategory $\mathrm{Add}(\mathrm{mod}\mbox{-}\Lambda)$ of $\mathrm{mod}\mbox{-}\Lambda$: 

Objects: arbitrary direct sums of objects of $\mathrm{mod}\mbox{-}\Lambda$.

Morphisms: $f:=(f_{ij}): \bigoplus_{j\in J} M_j \to \bigoplus_{i\in I} N_i$ where, $f_{ij}: M_j \to N_i$, for $i\in I, j \in J$. 

For any $j \in J$, we have $f_{ij}\neq 0$ only for a finite number of $i \in I$. Indeed, this follows from the fact that $M_j$ is finite-dimensional for any $j \in J$.

Given $T \in \mathcal{F}(\Lambda)$, define $\hat{T}: \mathrm{Add}(\mathrm{mod}\mbox{-}\Lambda) \to \mathrm{mod}\mbox{-} K$ as follows: 

Set $\hat{T}(\bigoplus_{j\in J} M_j) = \bigoplus_{j\in J} T(M_j)$ and $\hat{T}(f): \bigoplus_{i\in I} T(N_i) \to \bigoplus_{j\in J} T(M_j)$ is defined by $T(f_{ij}): T(N_i) \to T(M_j)$, for $i\in I, j \in J$. Observe that $\hat{T}$ equals $T$ on $\mathrm{mod}\mbox{-}\Lambda$.

\vspace{0.1in}
\noindent{\textbf{The functor $\phi$ between the functor categories over $\Lambda$ and $\bar{\Lambda}$:}}

Consider the functor $\phi: \mathcal{F}(\Lambda) \to \mathcal{F}(\bar\Lambda)$ as follows: 

$\phi(T):= \hat{T}\circ F^\lambda$ for any $T\in \mathcal{F}(\Lambda)$. Clearly, $\phi(T)\in \mathcal{G}(\bar\Lambda)$. 

For the semi-dense property of $F_\lambda$ as described in \cite[Corollary 3.14]{GoSaTr25}, given $\bar{M} \in \mathrm{Ind}\mbox{-}\bar\Lambda$ there exists a unique $M \in \mathrm{Ind}\mbox{-}\Lambda$, thanks to Remark \ref{indsk}, such that the following cases hold: 

\begin{itemize}
\item If $\bar{M}\approx F_\lambda(M)$, then $\phi(T)(\bar{M})=\hat{T}(F^\lambda(F_\lambda(M)))\approx \hat{T}(\bigoplus_{g\in G} {}^gM)= \bigoplus_{g\in G}T({}^gM).$

\item If $\bar{M}\inplus F_\lambda(M)$, then $F^\lambda(\bar{M})=M$. Hence we get $\phi(T)(\bar{M})= T(M).$ 
\end{itemize}

Define $\phi(T)(\alpha)= T(F^\lambda(\alpha))$ for any morphism $\alpha\in \mathrm{mod}\mbox{-}\bar{\Lambda}$. 


The next remark ensures that $\phi$ sends a homomorphism in $\mathcal{F}(\Lambda)$ to a homomorphism in $\mathcal{F}(\bar\Lambda)$.
\begin{rmk}
Assume that $T_1, T_2\in \mathcal{F}(\Lambda)$ and $\iota: T_1 \to T_2$ is a homomorphism of functors. Then $\phi(\iota): \phi(T_1) \to \phi(T_2)$ is defined as follows: 

Assume $\bar{M}, \bar{N} \in \mathrm{mod}\mbox{-}\bar\Lambda$ and $\bar{f} \in \mathrm{Hom}_{\bar\Lambda}(\bar{M}, \bar{N})$. There are two cases: 
\begin{itemize}
\item[$\bar{M}\inplus F_\lambda(M)$] Here, $\phi(\iota)_{\bar{M}}: T_1(M) \to T_2(M)$ is the homomorphism $\iota_{M}: T_1(M) \to T_2(M)$.
\item[$\bar{M}= F_\lambda(M)$] Set $\phi(\iota)_{\bar{M}}: \bigoplus_{g\in G}T_1({}^gM) \to \bigoplus_{g\in G}T_2({}^gM)$ to be the homomorphism defined by $\iota_{{}^gM}: T_1({}^gM) \to T_2({}^gM)$, for $g\in G$. 
\end{itemize}

Let us consider the following cases:
\begin{enumerate}
\item If $\bar{M}\inplus F_\lambda(M),\bar{N}\inplus F_\lambda(N)$, then $F^\lambda(\bar{f}_{\bar{M}\bar{N}})=f_{MN}$ where $f:M\to N$ is the associated homomorphism in $\mathrm{mod}\mbox{-}\Lambda$. Since, $\iota_MT_1(f)=T_2(f)\iota_N$ holds as $\iota$ is a homomorphism of functors, we get $\phi(\iota)_{\bar{M}}\phi(T_1)(\bar{f})=\phi(T_2)(\bar{f})\phi(\iota)_{\bar{N}}$.
    
\item If $\bar{M}\inplus F_\lambda(M),\bar{N}= F_\lambda(N)$, then $F^\lambda(\bar{f})=\uwithtext{$M$}_{g\in G} f_{M {}^gN}$ where $f_i:M\to {}^{g_i} N$ is the associated homomorphism in $\mathrm{mod}\mbox{-}\Lambda$. Since, $\iota_MT_1(f_i)=T_2(f_i)\iota_{{}^{g_i}N}$ holds as $\iota$ is a homomorphism of functors, we get $\phi(\iota)_{\bar{M}}\phi(T_1)(\bar{f})=\phi(T_2)(\bar{f})\phi(\iota)_{\bar{N}}$.
    
\item If $\bar{M}= F_\lambda(M),\bar{N}\inplus F_\lambda(N)$, then $F^\lambda(\bar{f})=\uwithtext{$N$}_{g\in G} f_{{}^gM N}$ where $f_i:g_i M\to N$ is the associated homomorphism in $\mathrm{mod}\mbox{-}\Lambda$. Since, $\iota_{g_iM}T_1(f_i)=T_2(f_i)\iota_N$ holds as $\iota$ is a homomorphism of functors, we get $\phi(\iota)_{\bar{M}}\phi(T_1)(\bar{f})=\phi(T_2)(\bar{f})\phi(\iota)_{\bar{N}}$.

\item If $\bar{M}=F_\lambda(M), \bar{N}= F_\lambda(N)$, then $F^\lambda(\bar{f})=\bigoplus_{g\in G} f_{{}^gM{}^gN}$ where $g_if:g_iM\to g_iN$ is the associated homomorphism in $\mathrm{mod}\mbox{-}\Lambda$. Since, $\iota_{{}^{g_i}M}T_1({}^{g_i}f)=T_2({}^{g_i}f)\iota_{{}^{g_i}N}$ holds as $\iota$ is a homomorphism of functors, we get $\phi(\iota)_{\bar{M}}\phi(T_1)(\bar{f})=\phi(T_2)(\bar{f})\phi(\iota)_{\bar{N}}$.
\end{enumerate}
This shows that, $\phi(\iota): \phi(T_1)\to \phi(T_2)$ is a homomorphism of functors. 
\end{rmk}

The next result ensures the exactness of $\phi$.

\begin{proposition}\label{exctphi}
$\phi: \mathcal{F}(\Lambda) \to \mathcal{G}(\bar\Lambda)$ is an exact functor.
\end{proposition}
\begin{proof}
For a given exact sequence $0\to T \xrightarrow{\iota} T' \xrightarrow{\iota'} T''\to 0$ and $\bar{M}\in \mathrm{mod}\mbox{-}\bar{\Lambda}$, the sequence $0\to T({}^{g_i}M) \xrightarrow{\iota_{{}^{g_i}M}} T'({}^{g_i}M) \xrightarrow{\iota'_{{}^{g_i}M}} T''({}^{g_i}M)\to 0$ is exact for all $g_i\in G$. Thus $\phi$ is exact by definition.   
\end{proof}

Below we show that $\phi$ respects the finitely presented functors, i.e., $\phi(T)\in \mathcal{F}(\bar\Lambda)$ for any $T\in \mathcal{F}(\Lambda)$.

\begin{proposition}\label{cokphi}
Suppose $M, N \in \mathrm{mod}\mbox{-}\Lambda$ and $f: M \to N$ is a homomorphism. Then
\begin{enumerate}
\item $\phi(\mathrm{Hom}_\Lambda(-, f))$ is isomorphic with $\mathrm{Hom}_\Lambda(F^\lambda(-), f): \mathrm{Hom}_\Lambda(F^\lambda(-),M)\to \mathrm{Hom}_\Lambda(F^\lambda(-), N)$.
\item $\phi(T)= \mathrm{Coker}\mathrm{Hom}_{\bar\Lambda}(-, F_\Lambda(f))\in \mathcal{F}(\bar\Lambda)$ for any $T= \mathrm{Coker}\mathrm{Hom}_\Lambda(-,f)$. Consequently, there is the following commutative diagram (up to isomorphism of functors):

\[
\begin{tikzcd}
\mathrm{H}(\mathrm{mod}\mbox{-}\Lambda) \arrow[rrr, "\mathcal{\theta}(\mathrm{mod}\mbox{-}\Lambda)"] \arrow[dd, "\mathrm{H}F_\lambda"'] &  &  & \mathcal{F}(\Lambda) \arrow[dd, "\phi"] \\
                                                                                                                                          &  &  &                                                             \\
\mathrm{H}(\mathrm{mod}\mbox{-}\bar\Lambda) \arrow[rrr, "\mathcal{\theta}(\mathrm{mod}\mbox{-}\bar\Lambda)"']                            &  &  & \mathcal{F}(\bar\Lambda)               
\end{tikzcd}
\]

where, $\mathcal{\theta}(\mathrm{mod}\mbox{-}\Lambda): \mathrm{H}(\mathrm{mod}\mbox{-}\Lambda)\to \mathcal{F}(\Lambda)$ is the Yoneda functor on objects of $\mathrm{H}(\mathrm{mod}\mbox{-}\Lambda)$ defined by $(X_1\xrightarrow{f} X_2)\to \mathrm{Coker}(\mathrm{Hom}_\Lambda(-, X_1)\xrightarrow{\mathrm{Hom}_\Lambda(-,f)} \mathrm{Hom}_\Lambda(-, X_2))$.
\end{enumerate} 
\end{proposition}
\begin{proof}
(1) For any $\bar{X}\in \mathrm{mod}\mbox{-}\bar{\Lambda}$, we have $\phi(\mathrm{Hom}_\Lambda(-, f))(\bar{X})=\mathrm{Hom}_\Lambda(-, f)\circ F^\lambda(\bar{X})=\mathrm{Hom}_\Lambda(F^\lambda(\bar{X}), f):\bigoplus_{g\in G} \mathrm{Hom}_\Lambda({}^gX, M) \to \bigoplus_{g\in G} \mathrm{Hom}_\Lambda({}^gX, N)$ is defined by homomorphism $\mathrm{Hom}_\Lambda({}^gX, f): \mathrm{Hom}_\Lambda({}^gX, M) \to \mathrm{Hom}_\Lambda({}^gX, N)$, for all $g \in G$ by Remark \ref{indsk}. Hence, the result. 


(2) For any $T= \mathrm{Coker}\mathrm{Hom}_\Lambda(-,f)$, there is an exact sequence $\mathrm{Hom}_\Lambda(-, M) \xrightarrow{\mathrm{Hom}_\Lambda(-,f)} \mathrm{Hom}_\Lambda(-, N) \to T \to 0$. As $\phi: \mathcal{F}(\Lambda) \to \mathcal{G}(\bar\Lambda)$ is exact by Proposition \ref{exctphi}, the following sequence of functors is exact: $$\phi(\mathrm{Hom}_\Lambda(-, M)) \xrightarrow{\phi(\mathrm{Hom}_\Lambda(-,f))} \phi(\mathrm{Hom}_\Lambda(-, N)) \to \phi(T) \to 0.$$  Moreover, $(F^\lambda, F_\lambda)$ is an adjoint pair by Proposition \ref{gcad}, thus the following diagram is commutative.
\[
\begin{tikzcd}
{\mathrm{Hom}_\Lambda(F^\lambda(-),M)} \arrow[rr, "{\mathrm{Hom}_\Lambda(F^\lambda(-),f)}"] \arrow[d, "\approx"'] &  & {\mathrm{Hom}_\Lambda(F^\lambda(-),N)} \arrow[d, "\approx"] \\
{\mathrm{Hom}_{\bar\Lambda}(-,F_\lambda(M))} \arrow[rr, "{\mathrm{Hom}_{\bar\Lambda}(-,F_\lambda(f))}"]           &  & {\mathrm{Hom}_{\bar\Lambda}(-,F_\lambda(N))}               
\end{tikzcd}
\]

Now by $(1)$, $\phi(T)\approx \mathrm{Coker}\phi(\mathrm{Hom}_\Lambda(-, f))\approx \mathrm{Coker}\mathrm{Hom}_\Lambda(F^\lambda(-), f)\approx \mathrm{Coker}\mathrm{Hom}_{\bar\Lambda}(-, F_\lambda(f))$. 
\end{proof}

Thus we get a covariant exact functor $\phi: \mathcal{F}(\Lambda) \to \mathcal{F}(\bar\Lambda)$ satisfying $\phi(T)= \hat{T}\circ F^\lambda\approx \mathrm{Coker}\mathrm{Hom}_{\bar\Lambda}(-, F_\lambda(f))$ for any $T = \mathrm{Coker}\mathrm{Hom}_\Lambda(-, f)$. In particular, $\phi(\mathrm{Hom}_\Lambda(-, M))= \mathrm{Hom}_{\bar\Lambda}(-, F_\lambda(M))$ and $\phi(\mathrm{Hom}_\Lambda(-, f))= \mathrm{Hom}_{\bar\Lambda}(-, F_\lambda(f))$ for any $M, N\in \mathrm{mod}\mbox{-}\Lambda$ and a homomorphism $f: M \to N$.

\vspace{0.1in}
\noindent{\textbf{The $G$-action on $\mathcal{G}(\Lambda)$:}} For any $T\in \mathcal{G}(\Lambda)$ and $g \in G$, define ${}^gT \in \mathcal{G}(\Lambda)$ as follows: $$({}^gT)(M) = T({}^{g^{-1}}M) \mbox{ for any } M \in \mathrm{mod}\mbox{-}\Lambda.$$ $$({}^gT)(f) = T({}^{g^{-1}}f) \mbox{ for any homomorphism } f \in \mathrm{mod}\mbox{-}\Lambda.$$ There are isomorphisms ${}^g\mathrm{Hom}_\Lambda(-, M)\approx \mathrm{Hom}_\Lambda(-, {}^gM)$ and ${}^g\mathrm{Coker}\mathrm{Hom}_\Lambda(-, f)\approx \mathrm{Coker}\mathrm{Hom}_\Lambda(-, {}^gf)$, for any $M \in \mathrm{mod}\mbox{-}\Lambda$, $f \in \mathrm{mod}\mbox{-}\Lambda$ and $g\in G$, and thus the action of $G$ on $\mathcal{G}(\Lambda)$ restricts to $\mathcal{F}(\Lambda)$. The following lemma shows that the stability of $T$ is equivalent to the simultaneous stability of the source and target of $f$ for any $T = \mathrm{Coker}\mathrm{Hom}_\Lambda(-, f) \in \mathcal{F}(\Lambda)$.

\begin{lemma}\label{sfm}
Suppose $T= \mathrm{Coker}\mathrm{Hom}_\Lambda(-, f)\in \mathcal{F}(\Lambda)$ for a $\Lambda$-module homomorphism $f:M\to N$ then $G_T=G$ if and only if $G_{MN}=G$.    
\end{lemma}
\begin{proof}
Here, $G_T=G$ implies $\mathrm{Coker}\mathrm{Hom}_\Lambda(-, {}^gf)\approx \mathrm{Coker}\mathrm{Hom}_\Lambda(-,f)$ for all $g\in G$ and hence $G_f=G$. Therefore, $G_{MN}=G$. The converse is trivial.
\end{proof}

The next theorem establishes the existence of the Galois semi-covering functor $\phi: \mathcal{F}(\Lambda) \to \mathcal{F}(\bar\Lambda)$.
\begin{theorem}\label{GCF}
Suppose a finite abelian group $G$ acts on a finite-dimensional algebra $\Lambda$ with $\bar\Lambda$ the associated skew group algebra. Then for any $T_1, T_2 \in \mathcal{F}(\Lambda)$, the functor $\phi: \mathcal{F}(\Lambda) \to \mathcal{F}(\bar\Lambda)$ induces the following isomorphisms of vector spaces:
\begin{equation}\label{GCFEQ}
\mathcal{F}(\bar\Lambda)(\phi(T_1), \phi(T_2))\approx \begin{cases}\bigoplus_{g\in G} \mathcal{F}(\Lambda)({}^gT_1, T_2)&\mbox{ if } G_{T_1}\neq G;\\\bigoplus_{g\in G} \mathcal{F}(\Lambda)(T_1, {}^gT_2)&\mbox{ if }G_{T_2}\neq G;\\\mathcal{F}(\Lambda)^{|G|}(T_1, T_2)&\mbox{ if }G_{T_1T_2}= G.\end{cases} 
\end{equation}
\end{theorem}
\begin{proof}
Assume that $T_i = \mathrm{Coker}\mathrm{Hom}_\Lambda(-,f_i)$ for $i=1,2$. If $G_{T_1}\neq G$, then $G_{f_1}\neq G$ by Lemma \ref{sfm}. From Proposition \ref{cokphi} and Theorem \ref{HGCM}, we get $\bigoplus_{g\in G} \mathcal{F}(\Lambda)({}^gT_1, T_2)\approx\bigoplus_{g\in G} \mathrm{H}(\mathrm{mod}\mbox{-}\Lambda)({}^gf_1, f_2)\approx \mathrm{H}(\mathrm{mod}\mbox{-}\bar\Lambda)(\mathrm{H}F_\lambda f_1,\mathrm{H}F_\lambda f_2)\approx \mathcal{F}(\bar\Lambda)(\phi(T_1), \phi(T_2))$. The remaining cases are similar.
\end{proof}

The following corollary ensures the faithfulness of $\phi$.
\begin{corollary}\label{faithful}
Assume that $G$ acts on an algebra $\Lambda$ and $\bar\Lambda$ is the associated skew group algebra. Then the functor $\phi: \mathcal{F}(\Lambda) \to \mathcal{F}(\bar\Lambda)$ is faithful. 
\end{corollary}
\begin{proof}
Composing the isomorphism in Theorem \ref{GCF} with the inclusion of $\mathcal{F}(\Lambda)(T_1, T_2)$ into the morphism space of the functor category on the right hand side of equation \ref{GCFEQ} in different cases (for example,  $\mathcal{F}(\Lambda)(T_1, T_2)\hookrightarrow{} \bigoplus \mathcal{F}(\Lambda)({}^gT_1, T_2)$ if $G_{T_1}\neq G$), we obtain a monomorphism from $\mathcal{F}(\Lambda)(T_1, T_2)$ to $\mathcal{F}(\bar\Lambda)(\phi(T_1), \phi(T_2))$. This shows that, $\phi$ is faithful.
\end{proof}

\section{KG-dimension of the skew group algebras}

The main contribution of this section is to demonstrate that KG dimension is preserved under skewness. As a result, we verify Prest's conjecture on the finiteness of KG dimension and Schr\"oer's conjecture on its connection with the stable rank over skew gentle algebras. Moreover, we determine all posible stable ranks for (skew) Brauer graph algebras. 

Here, we define $KG(\Lambda)$ as the KG dimension $KG(\mathcal{F}(\Lambda))$ of the category $\mathcal{F}(\Lambda)$ as follows.

\begin{definition}
Let $\mathcal{F}(\Lambda)$ be the category of all finitely presented contravariant $K$-linear functors from $\mathrm{mod}\mbox{-}\Lambda$ to $\mathrm{mod}\mbox{-} K$. A natural approach to study the structure of $\mathcal{F}(\Lambda)$ is via the associated Krull-Gabriel filtration:
$$\mathcal{F}(\Lambda)_{-1}\subseteq \mathcal{F}(\Lambda)_{0}\subseteq \mathcal{F}(\Lambda)_{1}\subseteq \hdots \subseteq \mathcal{F}(\Lambda)_{\alpha} \subseteq \mathcal{F}(\Lambda)_{\alpha+1} \subseteq \hdots$$ of $\mathcal{F}(\Lambda)$ by Serre subcategories, defined recursively as follows:

\begin{enumerate}
\item $\mathcal{F}(\Lambda)_{-1}=0$;
\item $\mathcal{F}(\Lambda)_{\alpha+1}$ is the Serre subcategory of $\mathcal{F}(\Lambda)$ formed by all functors having a finite length in the quotient category $\mathcal{F}(\Lambda)/\mathcal{F}(\Lambda)_{\alpha}$ for any ordinal number $\alpha$;
\item $\mathcal{F}(\Lambda)_{\beta}=\bigcup_{\alpha\textless\beta}\mathcal{F}(\Lambda)_{\alpha}$ for any limit ordinal $\beta$.
\end{enumerate}
The Krull-Gabriel dimension $KG(\Lambda)$ of $\Lambda$ is the smallest ordinal number $\alpha$ such that $\mathcal{F}(\Lambda)_{\alpha}=\mathcal{F}(\Lambda)$, if such a number exists, and $KG(\Lambda)=\infty$ if it is not the case. In the latter case, we call the smallest ordinal number $\alpha$ such that $\mathcal{F}(\Lambda)_{\alpha}=\mathcal{F}(\Lambda)_{\alpha+1}$ as the stable Krull-Gabriel dimension $KG_{st}(\Lambda)$ of $\Lambda$.
\end{definition}

The following lemma is essential to determine the KG dimension of a skew algebra.
\begin{lemma}\label{fapreservekg}\cite[Appendix~B]{Krau01}
Assume that $C, D$ are abelian categories and $F: C\to D$ is an exact functor. If $F$ is faithful, then $KG(C) \leq KG(D)$.
\end{lemma}


The next theorem concludes that the KG-dimension is preserved under skewness.

\begin{theorem}\label{eqkg}
Suppose $\Lambda$ is a finite-dimensional algebra, $G$ is a finite abelian group of $\mathcal K$-linear automorphisms of $\Lambda$ and $\bar\Lambda$ is the associated skew group algebra. Then, the KG dimension $KG(\bar\Lambda)$ of $\bar\Lambda$ equals the KG dimension $KG(\Lambda)$ of $\Lambda$.
\end{theorem}
\begin{proof}
Propositions \ref{exctphi} and \ref{faithful} ensure that $\phi$ is an exact faithful functor and hence $KG(\Lambda) \leq KG(\bar\Lambda)$ by Lemma \ref{fapreservekg}. Therefore, 
$KG(\Lambda) \leq KG(\bar\Lambda)\leq KG((\bar\Lambda)\hat{G})\approx KG(\Lambda)$ by  Theorem \ref{inv}.
\end{proof}

\subsection{KG-dimension of skew gentle algebras}
Given a quiver $Q= (Q_0,Q_1,s,t)$ and an admissible ideal $I$ of $KQ$, the algebra $KQ/I$ is called a special biserial algebra if 
\begin{enumerate}
\item at every vertex $v$ in $Q$ there are at most two arrows starting at $v$ and there are at most two arrows ending at $v$;
\item for every arrow $\alpha$ in $Q$ there exists at most one arrow $\beta$ such that $\alpha \beta\notin I$ and there
exists at most one arrow $\gamma$ such that $\gamma\alpha \notin I$.
\end{enumerate}

If the ideal $I$ is generated by paths of length at least two, then it is known as a string algebra. Gentle algebras form a subclass of string algebras for which the ideal $I$ is generated by paths of length precisely two.

For the detailed concept of gentle and skew-gentle algebras follow \cite[Section~4.1]{GoSaTr25}. Moreover, skew-gentle algebras are found to be the skew-group algebras of gentle algebras equipped with a certain $\mathbb{Z}_2$-action \cite[Section~3.2]{AB22}. As a direct application of Theorem \ref{eqkg} we get the following result:

\begin{corollary}\label{eqkgg}
Suppose $\bar\Lambda$ is the associated skew-gentle algebra of a gentle algebra $\Lambda$ then their KG dimensions are equal.   
\end{corollary}

The Prest's conjecture in Corollary \ref{conjsp} is valid for string algebras \cite{LaPrPu18}. Gentle algebras being a subclass of string algebras, the above corollary confirms the conjecture for skew gentle algebras as there are infinitely many bands that parameterize their non-domesticity.
\begin{corollary}\label{domfi}
Suppose $\bar\Lambda$ is a skew-gentle algebra. Then $\bar{\Lambda}$ is of domestic representation type if and only if the Krull-Gabriel dimension $KG(\bar{\Lambda})$ of $\bar{\Lambda}$ is finite.    
\end{corollary}

On the otherhand, the Schr\"oer's conjecture in Corollary \ref{conjsp} bridges between the KG dimension and the nilpotent index of the radical. Let us recall the definition of radical of an additive category. 
\begin{definition}\label{ranstab}
The radical $\mathrm{rad}_\Lambda$ of $\Lambda$ is the ideal generated by the non-invertible morphisms between indecomposable objects. A morphism in $\Lambda$ is called radical if it lies in $\mathrm{rad}(\Lambda)$. Its powers are defined inductively as follows.
\begin{enumerate}
    \item $\mathrm{rad}^n_\Lambda= \mathrm{rad}^{n-1}_\Lambda\mathrm{rad}_\Lambda$ if $n$ is finite;
    \item $\mathrm{rad}^\alpha_\Lambda:= \bigcap_{\mu<\alpha} \mathrm{rad}^\mu_\Lambda$ if $\alpha$ is a limit ordinal;
    \item $\mathrm{rad}^\alpha_\Lambda:= (\mathrm{rad}^\mu_\Lambda)^{n+1}$ if $\alpha= \mu + n$ is a successor ordinal;
    \item $\mathrm{rad}^\infty_\Lambda:= \bigcap_\mu \mathrm{rad}^\mu_\Lambda$.
\end{enumerate}

There is a descending chain of ideals
$\mathrm{mod}\mbox{-}\Lambda\supseteq\mathrm{rad}_\Lambda\supseteq\mathrm{rad}^2_\Lambda\supseteq\hdots\supseteq\mathrm{rad}^\omega_\Lambda\supseteq\mathrm{rad}^{\omega+1}_\Lambda\supseteq\hdots\supseteq\mathrm{rad}^\infty_\Lambda\supseteq0.$ The \emph{stable rank} of $\Lambda$, denoted by $\mathrm{st}(\Lambda)$, is the minimum ordinal $\alpha$ such that $\mathrm{rad}^\alpha_\Lambda=\mathrm{rad}^{\alpha+1}_\Lambda$. 
\end{definition}
The author with S. Trepode and A. G. Chaio prove that a Galois semi-covering functor preseves the power of radical \cite[Theorems~4.3, 4.6]{GoSaTr25}.
\begin{theorem}\label{psr}
Let $F_{\lambda}: \Lambda \to \bar{\Lambda}$ be a Galois semi-covering. Then 
\begin{enumerate}
\item $F_{\lambda}$ preserve powers of radicals;
\item stable rank of $\Lambda$= stable rank of $\bar{\Lambda}$.
\end{enumerate}
\end{theorem}

Moreover, the Schr\"oer's conjecture is verified in \cite[Theorem~10.19]{LaPrPu18} for string algebras. As a consequence of Theorems \ref{eqkg} and \ref{psr}, we get the following result:

\begin{corollary}\label{conjs}
For a skew-gentle algebra $\bar\Lambda$, $KG(\bar{\Lambda})= n \geq 2$ iff $\operatorname{rad}^{\omega(n-1)}_{\bar{\Lambda}}=0$ and $\operatorname{rad}^{\omega n}_{\bar{\Lambda}}=0$.
\end{corollary}

\subsection{KG-dimension of (skew) Brauer graph algebras}
The class of Brauer graph algebras forms the subclass of symmetric algebras of special biserial algebras as described in \cite[Theorem 2.8, Theorem 1.8]{Sc18}. For the detailed concept of Brauer graph algebras follow \cite[Section~2.3]{Sc18}.

\begin{theorem}\label{BSB}
Brauer graph algebras are symmetric special biserial.    
\end{theorem}
On the other hand, skew Brauer graph algebras are introduced as a generalization of Brauer graph algebras in \cite[Section~3.1]{So24} which are found to be the skew-group algebras of Brauer graph algebras equipped with a certain $\mathbb{Z}_2$-action described below.

\begin{definition}
A skew Brauer graph is the data $\Gamma = (H,\iota,\sigma,m)$ where
\begin{enumerate}
\item $H$ is the set of half-edges;
\item $\iota$ is a permutation of $H$ with possible fixed points satisfying $\iota^2= id_H$;
\item $\sigma$ is a permutation of $H$ called the orientation;
\item $m: H\to \mathbb{Z}_{>0}$ is a map that is constant on a $\sigma$-orbit : it is called the multiplicity.
\end{enumerate}
\end{definition}

One can naturally define a graph to the data $\Gamma = (H,\iota,\sigma,m)$. Denote $H_0$ the set of half-edges that are not fixed by $\iota$ and $H_\times$ the set of half-edges that are fixed by $\iota$. Hence, the vertex set of this graph is given by $\Gamma_0\approx H/\sigma \cup H_\times$ and the edge set is given by $\Gamma_1= H/\iota$. Moreover, each cycle in the decomposition of $\sigma$ gives rise to a cyclic ordering of the half-edges around a vertex in $H/\sigma$. Since $m: H\to \mathbb{Z}_{>0}$ is constant on a $\sigma$-orbit, it induces a map $m: H/\sigma \to \mathbb{Z}_{>0}$. We denote by $o$ the vertices in $H/\sigma$ and by $\times$ the vertices in $H_\times$. Clearly, each $\times$-vertex has a unique edge incident to it. We identify $\Gamma = (H,\iota,\sigma,m)$ with the newly constructed graph.
 
Given a skew Brauer graph $\Gamma = (H,\iota,\sigma,m)$, we say that a $\mathbb{Z}_2$-grading $d: H \to \mathbb{Z}_2$ is $0$-homogeneous if $\sum_{h\in H,s(h)=v} d(h)= 0$ for all $v\in H/\sigma$, where $s:H\to H/\sigma$ is the source map for $o$-vertex in $\Gamma$. Moreover, $(\Gamma,d)$ is a $\mathbb{Z}_2$-graded skew Brauer graph if $d: H \to \mathbb{Z}_2$ is $0$-homogeneous.

For a $\mathbb{Z}_2$-graded skew Brauer graph, construct a Brauer graph $\Gamma_d = (H_d,\iota_d,\sigma_d,m_d)$ as follows
\begin{enumerate}
\item $H_d = H\times\mathbb{Z}_2$: an element of $H_d$ will be denoted $h_i$ for $h\in H$ and $i\in \mathbb{Z}_2$;
\item For all $h_i\in H_d$, $\iota_d(h_i)=\begin{cases} h_{i+1} &\mbox{ if } h\in H_\times;\\(\iota h)_i &\mbox{ if } h\in H_0.\end{cases}$;
\item For all $h_i\in H_d, \sigma_d(h_i)= (\sigma h)_{i+d(h)}$;
\item For all $h_i\in H_d, m_d(h_i)= m(h)$.
\end{enumerate}

The skew Brauer graph algebra $B$ associated to $\Gamma$ can be found as a skew group algebra of the Brauer graph algebra $B_d$ associated to the covering $\Gamma_d$ under the following action of the cyclic group $G=\langle g\rangle$ of order $2$ on $B_d$. Denoting by $e_{[h_i]}$ the idempotent in $B_d$ corresponding to the edge $[h_i]\in H_d/\iota_d$, there is a natural action of $G$ on $B_d= KQ_d/I_d$ given by
\begin{enumerate}
\item $g. e_{[h_i]}= e_{[h_{i+1}]}$ for all $[h_i]\in H_d/\iota_d$;
\item $g. \alpha_{h_i}= \alpha_{h_{i+1}}$  where $\alpha_{h_i}$  is the arrow in $Q_d$ induced by the half-edge $h_i\in H_d$.
\end{enumerate}

Similarly, denoting by $e_{[h]_i}$  for $i= \phi,0,1$ the idempotent(s) in $B$ arising from the edge $[h]\in H/\iota$, there is a natural
action of the dual $\hat{G}=\langle \chi\rangle$ of $G$ on $B= KQ/I$ given by
\begin{enumerate}
\item $\chi.e_{[h]}= e_{[h]}$ for all $[h]\in H_o/\iota$ and $\chi. e_{[h]_i} = e_{[h]_{i+1}}$  for all $[h]\in H_\times/\iota$; 
\item $\chi.{}^j\alpha_h^i= (-1)^{-d(h)} {}^{j+1}\alpha_h^{i+1}$ where ${}^j\alpha_h^i$ is an arrow in $Q$ induced by the half-edge $h\in H$. By an abuse of notation, if $i=\phi$ then $i+1=\phi$.
\end{enumerate}

The following result \cite[Proposition~3.9]{So24} demonstrates that skew Brauer algebra is obtained as the skew group algebra of a Brauer algebra under the action of the group $G$ defined above.  

\begin{proposition}
Let $f$ be the idempotent in $B_d G$ given by the sum of the $e_{[h_0]} \otimes 1_G$ for all $[h] \in H/\iota$. Then, the algebra $fB_dGf$ has a natural $G$-action and the map $\phi: B \to fB_dGf$ defined by  $e_{[h]}\mapsto f[h]:= e_{[h_0]} \otimes 1_G$, $e_{[h]_i}\mapsto f_{[h]_i}:= e_{[h_0]} \otimes \frac{1_G+(-1)^ig}{2}$ for $i= 0,1$ and ${}^j\alpha_h^i\mapsto {}^j\beta_h^i$ for $i,j= \phi,0,1$ is an isomorphism of algebras which commutes with the $\hat{G}$-actions, where ${}^j\beta_h^i: f_{[h]_i}\to f_{[h]_j}$ are the arrows in $fB_d Gf$.    
\end{proposition} 

The following result is a direct consequence of Theorem \ref{eqkg} and the above proposition.
\begin{corollary}\label{eqkgb}
Suppose $\bar{\Lambda}$ is the skew Brauer graph algebra associated with the Brauer graph algebra $\Lambda$, then their KG dimensions are equal.
\end{corollary}

The next theorem by Kuber, Srivastava, and Sinha determines all possible stable ranks for special biserial algebras \cite{SVA23}.
\begin{theorem}\label{stablesba}
For any special biserial algebra $\Lambda$ with at least one band, $\omega \leq \mathrm{st}(\Lambda) < \omega^2$.    
\end{theorem}

Since Brauer graph algebras are special biserial by Theorem \ref{BSB}, the following result follows from Theorems \ref{psr} and \ref{stablesba}. 

\begin{theorem}\label{srbg}
For a skew Brauer graph algebra $\bar{\Lambda}$ and its associated Brauer graph algebra $\Lambda$, we have $\omega \leq \mathrm{st}(\Lambda)=\mathrm{st}(\bar{\Lambda}) < \omega^2$.  
\end{theorem}


\begin{thebibliography}{9999}
\bibitem{Au82} M. Auslander, A functorial approach to representation theory, in: Representations of Algebras, Lecture Notes in Math. Vol 944, 105–179, 1982.

\bibitem{AB22}  C. Amiot and T. Br\"ustle. Derived equivalences between skew-gentle algebras using
orbifolds. In: Documenta Mathematica 27, 933–982, doi: 10.4171/DM/889, 2022.

\bibitem{Ga80} S. Garavaglia, Dimension and rank in the model theory of modules, preprint, University of Michigan, 1979, revised 1980.
\bibitem{GdlP99} C. Geiss and J. A. de la Pe\~na, Auslander-Reiten components for clans, Bol. Soc. Mat. Mexicana 5 no. 2, 307–326, 1999.

\bibitem{Ge85}  W. Geigle, The Krull–Gabriel dimension of the representation theory of a tame hereditary artin algebra and applications to the structure of exact sequences, Manuscr. Math.,54, 83–106, 1985.

\bibitem{Ge86} W. Geigle, Krull dimension of Artin algebras, in: Representation Theory I. Finite-dimensional Algebras, Lecture Notes in Math., No. 1177, Springer-Verlag, Berlin, Heidelberg, New York, 135–155, 1986.

\bibitem{Ja69} G. Janusz, Indecomposable modules for finite groups. Ann. of Math. (2) 89, 209–241, 1969.

\bibitem{Krau01} H. Krause, The spectrum of a module category, Mem. Amer. Math. Soc. 149,  no. 707, 2001.

\bibitem{Kr98} H. Krause, Generic modules over artin algebras, Proc. London Math. Soc. 76, 276–306, 1998.

\bibitem{MaSc14} R. Marsh, S. Schroll, The geometry of Brauer graph algebras and cluster mutations, J. Algebra 419, 141–166, 2014.

\bibitem{LaPrPu18} R. Laking, M. Prest, G, Puninski, Krull–Gabriel dimension of domestic string algebras. Transactions of the American Mathematical Society, 370(7), 4813-4840, 2018.

\bibitem{Pr88} M. Prest, Model theory and modules, London Mathematical Society Lecture Note Series 130, Cambridge University Press, Cambridge, 1988.

\bibitem{Pr09} M. Prest, Purity, Spectra and Localization, Encyclopedia of Mathematics and Its Applications 121, Cambridge University Press, Cambridge, 2009.

\bibitem{Jose83} J. A. de la Pe\~na, Automorfismos, algebras torcidas y cubiertas, PhD thesis, UNAM, 1983.

\bibitem{ReRi85} I. Reiten and C. Riedtmann. Skew group algebras in the representation theory of Artin algebras. J. Algebra, 92(1):224–282, 1985.

\bibitem{So24}  V. Soto, Tilting mutations as generalized Kauer moves for (skew) Brauer graph algebras with multiplicity. arXiv preprint arXiv:2406.10634, 2024.

\bibitem{GoSaTr25} S. Sardar, A. Gonzalez and S. Trepode. On the irreducible Morphisms for skew Group Algebras, https://arxiv.org/abs/2507.20103, 2025.

\bibitem{Sc18} S. Schroll, Brauer graph algebras: a survey on Brauer graph algebras, associated gentle algebras and
their connections to cluster theory. In Homological methods, representation theory, and cluster algebras,
CRM Short Courses, 177–223. Springer, Cham, 2018.

\bibitem{SkWa83} A. Skowronski and J. Waschb\"usch. Representation-finite biserial algebras. In: Jour
nal f\"ur die reine und angewandte Mathematik 345, 172–181, 1983.

\bibitem{Sk16} A. Skowronski, The Krull-Gabriel Dimension of Cycle-Finite Artin Algebras, Algebr. Represent. Theory, Vol. 19 no. 1, 215–233, 2016.

\bibitem{Sch00} J. Schr\"oer, On the infinite radical of a module category, In: Proceedings of the London
Mathematical Society 81.3, 651–674, 2000.

\bibitem{Sc00} J. Schr\"oer, The Krull–Gabriel dimension of an algebra - open problems and conjectures,
pp. 419-424 in H. Krause and C. M. Ringel (Eds.), Infinite Length Modules, Birkh\"auser, 2000.

\bibitem{SVA23} S. Srivastava, V. Sinha, and A. Kuber, On the stable radical of the module category for special biserial algebras. In: arXiv:2311.10178, 2023.

\bibitem{WW85} B. Wald and J. Waschb\"usch, Tame biserial algebras, In: Journal of Algebra 95.2, 480–500, 1985.

\bibitem{We96}  M. Wenderlich, Krull dimension of strongly simply connected algebras, Bull. Polish Acad. Sci., Ser. Math., 44, 473–480, 1996.
\end{thebibliography}
\end{document}